\documentclass[a4paper]{article}
\usepackage{amssymb}
\usepackage{amsmath}
\usepackage{amsfonts}
\usepackage{mitpress}
\usepackage{geometry}

\newtheorem{theorem}{Theorem}

\newtheorem{corollary}[theorem]{Corollary}

\newtheorem{lemma}[theorem]{Lemma}

\newtheorem{proposition}[theorem]{Proposition}

\newtheorem{remark}[theorem]{Remark}

\newenvironment{proof}[1][Proof]{\noindent\textbf{#1.} }{\ \rule{0.5em}{0.5em}}

\newdimen\dummy
\dummy=\oddsidemargin
\input{tcilatex}
\begin{document}

\title{Marginal density expansions for diffusions and stochastic volatility,
part II:\ Applications}
\author{J.D. Deuschel, P.K. Friz, A. Jacquier, S. Violante \\
TU Berlin, TU and WIAS\ Berlin, TU\ Berlin, Imperial College}
\maketitle

\begin{abstract}
In \cite{DFJVpartI} we discussed density expansions for multidimensional
diffusions $\left( X^{1},\dots ,X^{d}\right) $, at fixed time $T$ and
projected to their first $l$ coordinates, in the small noise regime. Global
conditions were found which replace the well-known "not-in-cutlocus"
condition known from heat-kernel asymptotics. In the present paper we
discuss financial applications; these include tail and implied volatility
asymptotics in some correlated stochastic volatility models. In particular,
we solve a problem left open by A. Gulisashvili and E.M.\ Stein (2009).

\end{abstract}

\section{Introduction}

Given a multi-dimensional diffusion process $\mathrm{X}_{t}=\left(
X_{t}^{1},\dots ,X_{t}^{d}:t\geq 0\right) $, started at $\mathrm{X}_{0}=%
\mathrm{x}_{0}$, we studied in \cite{DFJVpartI} the behaviour of the
probability density function $f=f\left( \mathrm{y}\,,t\right) $ of the
projected (in general non-Markovian) process%
\begin{equation*}
\mathrm{Y}_{t}:=\Pi _{l}\circ \mathrm{X}_{t}:=\left( X_{t}^{1},\dots
,X_{t}^{l}\right) 
\end{equation*}%
with $l\in \left\{ 1,\dots ,d\right\} $ fixed. This situation is typical in
analysis of stochastic volatility models; $\mathrm{Y}$ may represent one (or
several!)\ assets, the full process $\mathrm{X}$ contains additional
stochastic volatility components (and also stochastic rates, if desired).
Basket models, in the spirit of \cite{A1, A2} can also be fitted in this
framework. Both short time asymptotics and tail asymptotics, in presence of
suitable scaling properties of the model, can be derived from the small
noise problem 
\begin{equation*}
d\mathrm{X}_{t}^{\varepsilon }=b\left( \varepsilon ,\mathrm{X}%
_{t}^{\varepsilon }\right) dt+\varepsilon \sigma \left( \mathrm{X}%
_{t}^{\varepsilon }\right) dW_{t},\quad \text{with }\mathrm{X}%
_{0}^{\varepsilon }=\mathrm{x}_{0}^{\varepsilon }\in \mathbb{R}^{d},
\end{equation*}%
where $W$ is a $m$-dimensional standard Brownian motion. The main technical
result in \cite{DFJVpartI} is a density expansion for $\mathrm{Y}%
_{t}^{\varepsilon }:=\Pi _{l}\circ \mathrm{X}_{t}^{\varepsilon }$ of the
form, for $\mathrm{x}_{0},\mathrm{y},T$ fixed, 
\begin{equation}
f^{\varepsilon }\left( \mathrm{y},T\right) =e^{-c_{1}/\varepsilon
^{2}}e^{c_{2}/\varepsilon }\varepsilon ^{-l}\left( c_{0}+O\left( \varepsilon
\right) \right) \text{ as }\varepsilon \downarrow 0.  \label{feps_intronew}
\end{equation}%
One of our main motivations, for  \cite{DFJVpartI} and the present paper, comes from the recent work on asset price
density expansions by A. Gulisashvili and E.M. Stein: in \cite[Theorem 2.1]{GuSt} they
consider the \textit{uncorrelated} Stein--Stein stochastic volatility
model.\ In the appropriate pricing measure the dynamics are\footnote{%
Sometimes the Stein--Stein model is written with $\left\vert Z\right\vert
dW^{1}$ rather than $ZdW^{1}$; in the uncorrelated case this does not make
a difference to the law of the process, as is immediate from a look at the respective generators.
There is a recent tendency in the finance community to use the form $ZdW^{1}$ which we analyze here, 
cf. \cite{LiSe, Li12}, this version of the model was also proposed by Sch\"obel--Zhu, \cite{SZ}.}
\begin{equation*}
dS/S=ZdW^{1},\,dZ=\left( a+bZ\right) dt+cdW^{2}\text{,}
\end{equation*}%
with parameters, $a\geq 0,b\leq 0,c>0$, spot-volatility $Z_{0}=\sigma
_{0}\geq 0$ and \textit{correlation }$\rho :=d\,\left\langle
W^{1},W^{2}\right\rangle /dt$ \textit{equal to zero}; spot $S_{0}$ can 
normalized to unit. Their main result is that $S_{T}$, for fixed $T>0$,
admits a probability density function $f=f\left( s\right) $ such that%
\footnote{%
Strictly speaking, the $O$-term given in \cite{GuSt} is $\log s$ with power $-1/4$; the authors
have informed us, however, that a closer look at their argument indeed gives
power $-1/2$.} 
\begin{equation}
f\left( s\right) = s^{-B_{1}}e^{B_{2}\sqrt{\log s}}\left( \log s\right)
^{-\frac{1}{2}}\left( B_0+O\left( \log s\right) ^{-\frac{1}{2}}\right) \text{
as }s\uparrow \infty   \label{ExpSS}
\end{equation}%
with explicitly computable constants; asymptotic formulae of the implied
volatility in the large strike regime are then obtained as (nowadays
mechanical; cf. Lee \cite{Lee11} and the references therein) corollaries.\
Indeed, one has\footnote{%
Small strike asymptotics are similar and will not be discussed here.} 
\begin{eqnarray}
\sigma _{BS}\left( k,T\right) ^{2}T &=&\left( \beta _{1}\sqrt{k}+\beta
_{2}+o\left( 1\right) \right) ^{2}\text{ as log-strike }k\rightarrow \infty ;
\label{ImplVolExp} \\
\beta _{1} &=&\beta _{1}\left( B_{1}\right) =\sqrt{2}\left( \sqrt{B_{1}-1}-%
\sqrt{B_{1}-2}\right) ,\text{ }  \notag \\
\beta _{2} &=&\beta _{2}\left( B_{1},B_{2}\right) =\frac{B_{2}}{\sqrt{2}}%
\left( \frac{1}{\sqrt{B_{1}-2}}-\frac{1}{\sqrt{B_{1}-1}}\right) .  \notag
\end{eqnarray}%
The proof of \cite[Theorem 2.1]{GuSt} relies on the so-called Hull-White
formula which states that - \textit{under the crucial assumption of zero
correlation} - an option price in a stochastic volatility models is
effectively a weighted average of Black--Scholes option prices (at different
volatiliy levels). The \textit{correlated case} was left as open problem in 
\cite[Theorem 2.1]{GuSt} and indeed the importance of allowing for
correlation in stochastic volatility models is well-documented, e.g. \cite%
{Ga06, LM}. Evidence from estimation of parametric stochastic volatility
models suggests correlation parameter $\rho \approx -0.7$ or $\rho \approx
-0.8$ for S\&P 500, for instance; a finding fairly robust across models and
time periods \cite{AS}.

\bigskip

When writing the expansion (\ref{ExpSS}) in terms of log-price $Y=\log S$,
it indeed has the form (\ref{feps_intronew}) with $y=\log s=1/\varepsilon
^{2}$ and $c_{1}=B_{1}-1,c_{2}=B_{2}$. More generally, we can show from
rather general and robust principles that the tail behaviour of $Y_{T}\in 
\mathbb{R}^{1}$ for fixed $T>0$, subject to a certain scaling with parameter 
$\theta \in \{1,2\}$ in the full Markovian specification of the model, has
the form 
\begin{equation}
f\left( y,T\right) =e^{-c_{1}y^{2/\theta }}e^{c_{2}y^{1/\theta }}y^{\frac{1}{%
\theta }-1}\left( c_{0}+O\left( 1/y^{1/\theta }\right) \right) \text{ as }%
y\uparrow \infty .  \label{flarge_intronew}
\end{equation}%
Again, such an expansion leads immediately to call price and then
(Black--Scholes) implied volatility expansions in the large strike regime,
cf. \cite{GuSt, Lee11}; in particular, in the case $\theta =2$ typical for
stochastic volatility (see \cite{FGGS} for similar results in the Heston
model) the expansion (\ref{ImplVolExp}) remains valid with $B_{1}=c_{1}+1$
and $c_{2}=B_{2}$. We note that the square-root growth of implied
volatility, in terms of log-strike, is actually a very general feature of
models with moment explosions, \cite{Lee, BF} which includes many stochastic
volatility models \cite{Ga06, LM, BF2}.

The main contribution of this paper is to establish validity of (\ref{ExpSS}%
), equivalently (\ref{flarge_intronew}) with $\theta =2$, for the \textit{%
correlated} Stein--Stein model. Having in mind the typical values o $\rho $
in equity markets, our focus is on the case $-1<\rho \leq 0$ (although our
analysis could be adapted to positive correlation). The leading
order behaviour described by $\beta _{1}=\beta _{1}\left( c_{1}+1\right) $ is
well understood; see \cite{Lee, BF} and also \cite[p40, p265]{Deuschel}. The
second order behaviour is given by $\beta _{2}=\beta _{2}\left(
c_{1}+1,c_{2}\right) $. Further terms in this expansion are in principle
possible \cite{Lee11}; in particular, the next term would involve $c_{0}$.
Our main observation is that the Stein--Stein model has the scaling
properties necessary to transform it into a small noise problem which can
then be tackled with the methods of \cite{DFJVpartI}. It should be noted
that the Stein--Stein model is hypoelliptic, with region of degeneracy given
by $\left\{ \left( y,z\right) :z=0\right\} $, and that the $\varepsilon $%
-rescaled Stein--Stein model is started (as $\varepsilon \rightarrow 0$) in
the degenerate region. In other words, there is no escape in dealing with
the hypoellipticity of the problem.\footnote{%
In contrast, short time asymptotics in "locally elliptic" stochastic
volatility models usually can be localized to a non-degenerate region. } 

Density expansions of diffusions in the small noise regime, with
applications to implied volatility expansions, were recently considered by
Y. Osajima \cite{O}, based on joint work with S. Kusuoka \cite{KO} and
old work of Kusuoka--Strook \cite{KS}. We
partially improve on these results. First, as was already mentioned in \cite{DFJVpartI},
any  expansion of the form (\ref{feps_intronew}), or (\ref{flarge_intronew}), with $c_2 \ne 0$ is out of reach in these works,
the reason being that the Kusuoka--Stroock theory was set up as expansion in $\varepsilon^2$
rather than $\varepsilon$. Secondly, in comparison with \cite{O}, we do not assume $%
\mathrm{x}_{0}$ near $\left( \mathrm{y},\cdot \right) $. And finally, in
further contrast to (the general results in) \cite{KO, KS} we provide a
checkable, finite-dimensional criterion that guarantees that the crucial
infinite-dimensional non-degeneracy assumption, left as such in \cite{KO, KS}%
, is actually satisfied. On the other hand, these authors give explicit
formulae for $c_{0}$ which we (presently) do not. Let us also emphasize (cf. corollary \ref{CorShortTime} below) that the expansion (\ref{feps_intronew}) can be 
used, as a simple consequence of Brownian scaling,
towards short time expansion for projected diffusion densities, under global
conditions on $\left( \mathrm{x}_{0},\mathrm{y}\right) $, of the form 
\begin{equation}
f\left( \mathrm{y},t\right) \sim   e^{ -\frac{d^{2}\left( \mathrm{x}_{0},%
\mathrm{y}\right) }{2t}} {t^{-l/2}}   c_{0}\left( \mathrm{x}_{0},\mathrm{y}%
\right)    \text{ as }t\downarrow 0.
\label{fshort_intronew}
\end{equation}%
When $l=d$, and then $\mathrm{y=x}$, such expansions go back to classical works starting with
Molchanov \cite{Mo75} (itself the main reference for the famous SABR\ paper, 
\cite{HLW}). The leading order behaviour $2t\log f\left( \mathrm{x},t\right)
\sim -d^{2}\left( \mathrm{x}_{0},\mathrm{x}\right) $ is due to Varadhan \cite%
{Va67}. The case $l<d$, in particular our global condition on $\left( 
\mathrm{x}_{0},\mathrm{y}\right) $, appears to be new. That said, expansions
of this form have appeared in \cite{TW, HLW, O}; the last two references
aimed at implied volatility expansions. In the context of a time-homogenuous
local volatility models ($l=d=1$), the expansion (\ref{fshort_intronew})
holds trivially without any conditions on $\left( \mathrm{x}_{0},\mathrm{y}%
\right) $; the resulting expansion was derived (with explicit constant $c_{0}
$) in \cite{GaEtAl}. Subject to mild technical conditions on the diffusion
coefficient, they show how to deduce first a call price and then an implied
volatility expansion in the short time (to maturity) regime:%
\begin{equation*}
\sigma _{BS}\left( k,t\right) =\left\vert k\right\vert /d\left( \mathrm{x}%
_{0},k\right) +c\left( \mathrm{x}_{0},k\right) t+O\left( t^{2}\right) \text{
as }t\downarrow 0;
\end{equation*}%
where $d\left( \mathrm{x}_{0},k\right) $ is a point-point distance and $%
c\left( \mathrm{x}_{0},k\right) $ is explicitly given. The celebrated
Berestycki--Busca--Florent (BBF) formula \cite{BBF} asserts that $\sigma
_{BS}\left( k,t\right) \sim \left\vert k\right\vert /d\left( \mathrm{x}%
_{0},k\right) $ as $t\downarrow 0$, is in fact valid in generic stochastic
volatilty models, $d\left( \mathrm{x}_{0},k\right) $ is then understood as
point-hyperplane distance. In fact, $\left\vert k\right\vert /d\left( 
\mathrm{x}_{0},k\right) $ arose as initial condition of a non-linear
evolution equation for the entire implied volatility surface. As briefly
indicated in \cite[Sec 6.3]{BBF} this can be used for a Taylor expansion of $%
\sigma _{BS}\left( k,t\right) $ in $t$. Such expansions have also been
discussed, based on heat kernel expansions on Riemannian manifolds by \cite%
{BC, HL, Pau}, not always in full mathematical rigor. Some mathematical
results are given in \cite{O}, assuming ellipticity and \textit{%
close-to-the-moneyness} $\left\vert k\right\vert <<1$; see also forthcoming
work by Ben Arous--Laurence \cite{BenArousLaurence}. We suspect that our
formula (\ref{fshort_intronew}), potentially applicable far-from-the-money,
will prove useful in this context and shall return to this in future work.

It should be noted, that the BBF formula alone can be obtained from soft
large deviation arguments, cf. \cite[Sec. 3.2.1]{Ph} and the references
therein. In a similar spirit,  cf. \cite[Sec 5, Rmk 2.9]{VLec}, the Varadhan-type formula $2t\log f\left( 
\mathrm{y},t\right) \sim -d^{2}\left( \mathrm{x}_{0},\mathrm{y}\right) $,
when $l<d$, could be shown, without any conditions on $(\mathrm{x}_{0},\mathrm{%
y})$ by large deviation methods, only relying on the existence of a
reasonable density.

As a final note, we recall that the (in general, non-Markovian) $\mathbb{R}%
^{l}$-valued It\^{o}-process $\left( \mathrm{Y}_{t}:t\geq 0\right) $ admits
- subject to some technical assumptions \cite{Gy, Pi} - a \textit{Markovian
(or Gy\"{o}ngy) projection}. That is, a time-inhomogeneous Markov diffusion $%
(\mathrm{\tilde{Y}}_{t}:t\geq 0)$ with matching time-marginals i.e $\mathrm{Y%
}_{t}=\mathrm{\tilde{Y}}_{t}$ \ (in law) for every fixed $t\geq 0$. In a
financial context, when $l=1$, this process is known as (Dupire) local
volatility model and various authors \cite{BBF, BC, HL, BenArousLaurence}
have used this as an important intermediate step in computing implied
volatility in stochastic volatility models. Since all our expansions (small
noise, tail, short time ) are relative to such time-marginals they may also
be viewed as expansions for the corresponding Markovian projections.

\ \textbf{Acknowledgement:} JDD and AJ acknowledge (partial resp. full)
finanical support from MATHEON. PKF\ acknowledges partial support from
MATHEON\ and the European Research Council under the European Union's
Seventh Framework Programme (FP7/2007-2013) / ERC grant agreement nr.
258237. PKF would like to thank G. Ben Arous for pointing out conceptual
similarities in \cite{FGGS, Benarous} and several discussions thereafter. It
is also a pleasure to thank F. Baudoin, J.P. Gauthier, A. Gulisashvili and
P. Laurence for their interest and feedback.

\section{The main result of \protect\cite{DFJVpartI}}

\bigskip Consider a $d$-dimensional diffusion $\left( \mathrm{X}%
_{t}^{\varepsilon }\right) _{t\geq 0}$ given by the stochastic differential
equation 
\begin{equation}
d\mathrm{X}_{t}^{\varepsilon }=b\left( \varepsilon ,\mathrm{X}%
_{t}^{\varepsilon }\right) dt+\varepsilon \sigma \left( \mathrm{X}%
_{t}^{\varepsilon }\right) dW_{t},\quad \text{with }\mathrm{X}%
_{0}^{\varepsilon }=\mathrm{x}_{0}^{\varepsilon }\in \mathbb{R}^{d}
\label{SDEXeps}
\end{equation}%
and where $W=(W^{1},\dots ,W^{m}$) is an $m$-dimensional Brownian motion.
Unless otherwise stated, we assume $b:[0,1)\times \mathbb{R}^{d}\rightarrow 
\mathbb{R}^{d},$ $\sigma =\left( \sigma _{1},\dots ,\sigma _{m}\right) :%
\mathbb{R}^{d}\rightarrow \mathrm{Lin}\left( \mathbb{R}^{m},\mathbb{R}%
^{d}\right) $ and $\mathrm{x}_{0}^{\cdot }:[0,1)\rightarrow \mathbb{R}^{d}$
to be smooth, bounded with bounded derivatives of all orders. Set $\sigma
_{0}=b\left( 0,\cdot \right) $ and assume that, for every multiindex $\alpha 
$, the drift vector fields $b\left( \varepsilon ,\cdot \right) $ converges
to $\sigma _{0}$ in the sense\footnote{%
If (\ref{SDEXeps}) is understood in Stratonovich sense, so that $dW$ is
replaced by $\circ dW$, the drift vector field $b\left( \varepsilon ,\cdot
\right) $ is changed to $\tilde{b}\left( \varepsilon ,\cdot \right) =b\left(
\varepsilon ,\cdot \right) -\left( \varepsilon ^{2}/2\right)
\sum_{i=1}^{m}\sigma _{i}\cdot \partial \sigma _{i}$. In particular, $\sigma
_{0}$ is also the limit of $\tilde{b}\left( \varepsilon ,\cdot \right) $ in
the sense of (\ref{bepsTob}) .}%
\begin{equation}
\partial _{x}^{\alpha }b\left( \varepsilon ,\cdot \right) \rightarrow
\partial _{x}^{\alpha }b\left( 0,\cdot \right) =\partial _{x}^{\alpha
}\sigma _{0}\left( \cdot \right) \text{ uniformly on compacts as }%
\varepsilon \downarrow 0\text{.}  \label{bepsTob}
\end{equation}%
We shall also assume that%
\begin{equation}
\partial _{\varepsilon }b\left( \varepsilon ,\cdot \right) \rightarrow
\partial _{\varepsilon }b\left( 0,\cdot \right) \text{ \ uniformly on
compacts as }\varepsilon \downarrow 0  \label{bepsC1}
\end{equation}%
and 
\begin{equation}
\mathrm{x}_{0}^{\varepsilon }=\mathrm{x}_{0}+\varepsilon \mathrm{\hat{x}}%
_{0}+o\left( \varepsilon \right) \text{ as }\varepsilon \downarrow 0\text{. }
\label{x0eps_ass}
\end{equation}

\begin{theorem}
\label{thm:MainThm} \textbf{(Small noise)} Let $\left( \mathrm{X}%
^{\varepsilon }\right) $ be the solution process to%
\begin{equation*}
d\mathrm{X}_{t}^{\varepsilon }=b\left( \varepsilon ,\mathrm{X}%
_{t}^{\varepsilon }\right) dt+\varepsilon \sigma \left( \mathrm{X}%
_{t}^{\varepsilon }\right) dW_{t},\quad \text{with }\mathrm{X}%
_{0}^{\varepsilon }=\mathrm{x}_{0}^{\varepsilon }\in \mathbb{R}^{d}.
\end{equation*}%
Assume $b\left( \varepsilon ,\cdot \right) \rightarrow \sigma _{0}\left(
\cdot \right) $ in the sense of (\ref{bepsTob}), (\ref{bepsC1}), and $%
\mathrm{X}_{0}^{\varepsilon }\equiv \mathrm{x}_{0}^{\varepsilon }\rightarrow 
\mathrm{x}_{0}$ as $\varepsilon \rightarrow 0$ in the sense of (\ref%
{x0eps_ass}). Assume the weak H\"{o}rmander condition (\ref{H}) at $\mathrm{x%
}_{0}\in \mathbb{R}^{d}$;%
\begin{equation}
\mathrm{span}\left[ \sigma _{i}:1\leq i\leq m;\text{ }\left[ \sigma
_{j},\sigma _{k}\right] :0\leq j,k\leq m;...\right] _{\mathrm{x}_{0}}=%
\mathbb{R}^{d};  \tag{H}  \label{H}
\end{equation}%
Fix $\mathrm{y}\in \mathbb{R}^{l},\,N_{\mathrm{y}}:=\left( \mathrm{y},\cdot
\right) $ and let $\mathcal{K}_{\mathrm{y}}$ be the the space of all $%
\mathrm{h}\in H$ s.t. the solution to%
\begin{equation*}
d\phi _{t}^{\mathrm{h}}=\sigma _{0}\left( \phi _{t}^{\mathrm{h}}\right)
dt+\sum_{i=1}^{m}\sigma _{i}\left( \phi _{t}^{\mathrm{h}}\right) d\mathrm{h}%
_{t}^{i},\,\,\phi _{0}^{\mathrm{h}}=\mathrm{x}_{0}\in \mathbb{R}^{d}
\end{equation*}%
satisfies $\phi _{T}^{\mathrm{h}}\in N_{\mathrm{y}}$. 
We assume $\mathcal{K}_{\mathrm{y}}$ to be non-empty%
\footnote{A well-known sufficient condition (cf. \cite{DFJVpartI} and the references
therein) is the \textit{strong H\"{o}rmander condition}  (H1), as stated in
corollary \ref{CorShortTime} below.}
 and the energy 
\begin{equation*}
\Lambda \left( \mathrm{y}\right) =\inf \left\{ \frac{1}{2}\Vert \mathrm{h}%
\Vert _{H}^{2}:\mathrm{h}\in \mathcal{K}_{\mathrm{y}}\right\} .
\end{equation*}%
to be a smooth function in a neighbourhood of $\mathrm{y}$.  Asssume furthermore \newline
(i) there are only finitely many minimizers, i.e. $\mathcal{K}_{%
\mathrm{y}}^{\min }<\infty $ where 
\begin{equation*}
\mathcal{K}_{\mathrm{y}}^{\min }:=\left\{ \mathrm{h}_{0}\in \mathcal{K}_{%
\mathrm{y}}:\frac{1}{2}\Vert \mathrm{h}_{0}\Vert _{H}^{2}=\Lambda \left( 
\mathrm{y}\right) \right\} ;
\end{equation*}%
(ii) non-degeneracy of the so-called deterministic Malliavin covariance
matrix at each minimizer - a sufficient condition met in most ("locally elliptic") financial
models reads%
\begin{equation*}
\forall \mathrm{h}_{0}\in \mathcal{K}_{\mathrm{y}}^{\min }:\exists t\in %
\left[ 0,T\right] :\mathrm{span}\left[ \sigma _{1},\dots ,\sigma _{m}\right]
|_{\phi _{t}^{\mathrm{h}_{0}}}=\mathbb{R}^{d};
\end{equation*}%
(iii) $\mathrm{x}_{0}$ is non-focal for $N_{\mathrm{y}}$ in the sense of 
\cite{DFJVpartI}. (We shall review below how to check this.)\newline

Then, keeping $\mathrm{x}_{0},$\textrm{y} and $T>0$ fixed, there exists $%
c_{0}=c_{0}\left( \mathrm{x}_{0},\mathrm{y},T\right) >0$ such that%
\begin{equation*}
\mathrm{Y}_{T}^{\varepsilon }=\Pi _{l}\mathrm{X}_{T}^{\varepsilon }=\left(
X_{T}^{\varepsilon ,1},\dots ,X_{T}^{\varepsilon ,l}\right) ,\,\,\,\,1\leq
l\leq d
\end{equation*}%
admits a density with expansion%
\begin{equation*}
f^{\varepsilon }\left( \mathrm{y},T\right) =e^{-\frac{\Lambda \left( \mathrm{%
y}\right) }{\varepsilon ^{2}}}e^{\,\frac{\max \left\{ \Lambda ^{\prime
}\left( \mathrm{y}\right) \cdot \,\mathrm{\hat{Y}}_{T}\left( \mathrm{h}%
_{0}\right) :\mathrm{h}_{0}\in \mathcal{K}_{\mathrm{y}}^{\min }\right\} }{%
\varepsilon }}\varepsilon ^{-l}\left( c_{0}+O\left( \varepsilon \right)
\right) \text{ as }\varepsilon \downarrow 0.
\end{equation*}%
Here $\mathrm{\hat{Y}=\hat{Y}}\left( \mathrm{h}_{0}\right) =\left( \hat{Y}%
^{1},\dots ,\hat{Y}^{l}\right) $ is the projection, $\mathrm{\hat{Y}=}\Pi
_{l}\mathrm{\hat{X}}$, of the solution to the following (ordinary)
differential equation 
\begin{eqnarray}
d\mathrm{\hat{X}}_{t} &=&\Big(\partial _{x}b\left( 0,\phi _{t}^{\mathrm{h}%
_{0}}\left( \mathrm{x}_{0}\right) \right) +\partial _{x}\sigma (\phi _{t}^{%
\mathrm{h}_{0}}\left( \mathrm{x}_{0}\right) )\mathrm{\dot{h}}_{0}\left(
t\right) \Big)\mathrm{\hat{X}}_{t}dt+\partial _{\varepsilon }b\left( 0,\phi
_{t}^{\mathrm{h}_{0}}\left( \mathrm{x}_{0}\right) \right) dt,
\label{eq:SDEYHat} \\
\,\,\,\,\,\mathrm{\hat{X}}_{0} &=&\mathrm{\hat{x}}_{0}.  \notag
\end{eqnarray}
\end{theorem}

\begin{remark}[Smoothness of energy]
If $\#\mathcal{K}_{\mathrm{y}}^{\min }=1$ smoothness of the energy is actually
a (non-trivial) consequence of the present assumptions and hence need not be assumed; 
\cite{DFJVpartI}. Note also that in our application to tail asymptotics, with $\theta $-scaling,
 $\theta \in \left\{1,2\right\} $ and scalar variable $y$, it follows from scaling that the energy 
 will be a linear resp. quadratic (and hence smooth) function of $y$. 
 \end{remark}

\begin{remark}[Localization]
The assumptions on the coefficients $b,\sigma $ in theorem \ref{thm:MainThm}
(smooth, bounded with bounded derivatives of all orders) are typical in this
context (cf. Ben\ Arous \cite{Benarous, Benarous2} for instance) but rarely
met in practical examples from finance. This difficulty can be resolved by a
suitable localization. For instance, as detailed in \cite{DFJVpartI}, an
estimate of the form  
\begin{equation}
\lim_{R\rightarrow \infty }\lim \sup_{\varepsilon \rightarrow 0}\varepsilon
^{2}\log \mathbb{P}\left[ \tau _{R}\leq T\right] =-\infty .  \label{Assloc}
\end{equation}%
with $\tau _{R}:=\inf \left\{ t\in \left[ 0,T\right] :\sup_{s\in \left[ 0,t%
\right] }\left\vert \mathrm{X}_{s}^{\varepsilon }\right\vert \geq R\right\} $
will allow to bypass the boundedness assumptions. 
\end{remark}

\section{\protect\bigskip Short time and tail asymptotics}

The reduction of \textit{short time expansions} to small noise expansions by
Brownian scaling is classical. In the present context, we have the following
statement, taken from \cite[Sec. 2.1]{DFJVpartI}. 

\begin{corollary}
\textbf{(Short time)} \label{CorShortTime}Consider $d\mathrm{X}_{t}=b\left( 
\mathrm{X}_{t}\right) dt+\sigma \left( \mathrm{X}_{t}\right) dW$, started at 
$\mathrm{X}_{0}=\mathrm{x}_{0}\in \mathbb{R}^{d}$, with $C^{\infty }$%
-bounded vector fields such that the strong H\"{o}rmander condition holds, 
\begin{equation}
\forall \mathrm{x}\in \mathbb{R}^{d}:\mathrm{Lie}\left[ \sigma _{1},\dots
,\sigma _{m}\right] |_{\mathrm{x}}=\mathbb{R}^{d}.  \tag{H1}  \label{H1cond}
\end{equation}%
Fix $\mathrm{y}\in \mathbb{R}^{l},\,N_{\mathrm{y}}:=\left( \mathrm{y},\cdot
\right) $ and assume (i),(ii),(iii) as in theorem \ref{thm:MainThm}. Let $%
f\left( t,\cdot \right) =f\left( t,\mathrm{y}\right) $ be the density of $%
\mathrm{Y}_{t}=\left( \mathrm{X}_{t}^{1},\dots ,\mathrm{X}_{t}^{l}\right) $.
Then, for some constant $c_0 = c_{0}\left( \mathrm{x}_{0},\mathrm{y}%
\right)>0$,
 \begin{equation}
f\left( \mathrm{y},t\right) \sim   e^{ -\frac{d^{2}\left( \mathrm{x}_{0},%
\mathrm{y}\right) }{2t}} {t^{-l/2}}   c_{0}  \text{ as }t\downarrow 0.
\end{equation}%
where $d\left( \mathrm{x}_{0},\mathrm{y}\right) $ is the sub-Riemannian
distance, based on $\left( \sigma _{1},\dots ,\sigma _{m}\right) $, from the
point $\mathrm{x}_{0}$ to the affine subspace $N_{\mathrm{y}}$.
\end{corollary}

We also have the following application to the \textit{tail behaviour} of,
say, the first component (i.e. $l=1$ here) of a diffusion processes at a
fixed time $T$. As we shall see, the scaling assumption below is met in a number of
stochastic volatility models.

\begin{corollary}
\label{CorTail}\bigskip \textbf{(Tail behaviour)} Assume $\mathrm{x}%
_{0}^{\varepsilon }\rightarrow 0\in \mathbb{R}^{d}$ as $\varepsilon
\rightarrow 0$ and some diffusion process $\mathrm{X}^{\varepsilon }$,
started at $\mathrm{x}_{0}^{\varepsilon }$, satisfies the assumptions of
theorem \ref{thm:MainThm} with $\mathrm{x}_{0}=0$ and $N=\left( 1,\cdot
\right) \subset \mathbb{R}\times \mathbb{R}^{d-1}$; in particular, $\left\{ 
\mathrm{0}\right\} \times \left( 1,\cdot \right) $ is assumed to satisfy
condition (i),(ii),(iii). Assume also $\theta $\textbf{-scaling} by which we
mean the scaling relation%
\begin{equation*}
Y_{T}^{\varepsilon }\overset{\text{(law)}}{=}\varepsilon ^{\theta }Y_{T}%
\text{ where \ \ }Y\equiv \Pi _{1}\mathrm{X}
\end{equation*}%
for some $\theta \geq 1$. Then the probability density function of $Y_{T}$
has the expansion 
\begin{equation}
f\left( y\right) =e^{-c_{1}y^{\frac{2}{\theta }}}e^{c_{2}y^{\frac{1}{\theta }%
}}y^{\frac{1}{\theta }-1}\left( c_{0}+O\left( 1/y^{1/\theta }\right) \right) 
\text{ as }y\rightarrow \infty   \label{densityTheta}
\end{equation}%
where  
\begin{eqnarray*}
c_{1} &=&\Lambda \left( 1\right)  \\
c_{2} &=&\hat{Y}_{T}\Lambda ^{\prime }\left( 1\right) =\frac{2\hat{Y}_{T}}{%
\theta }\Lambda \left( 1\right) 
\end{eqnarray*}%
and $c_{0}>0$. In particular, when $\theta =1$ we have a Gaussian tail
behaviour of the precise form%
\begin{equation*}
f\left( y\right) =e^{-\Lambda \left( 1\right) y^{2}}e^{2\hat{Y}_{T}\,\Lambda
\left( 1\right) y}\left( c_{0}+O\left( 1/y\right) \right) ;
\end{equation*}%
while $\theta =2$ leads to the exponential tail of the precise form%
\begin{equation*}
f\left( y\right) =e^{-\Lambda \left( 1\right) y}e^{\hat{Y}_{T}\,\Lambda
^{\prime }\left( 1\right) \sqrt{y}}y^{-1/2}\left( c_{0}+O\left( 1/\sqrt{y}%
\right) \right) .
\end{equation*}
\end{corollary}

\begin{proof}
Let $f^{\varepsilon }$ denote the density of $Y_{T}^{\varepsilon }$. Since $%
f\left( \cdot /\varepsilon ^{\theta }\right) =\varepsilon ^{\theta
}f^{\varepsilon }\left( \cdot \right) $ we can take  $\cdot =1 \in   \mathbb{R}^{l}$, with $l=1$, and
apply theorem \ref{thm:MainThm}. This yields the claimed expansion in the "large space" variable
$y = 1/ \varepsilon ^{\theta }$;  it suffices to rephrase the $\varepsilon$-expansion of theorem \ref{thm:MainThm}
in terms of $y$.  Another observation is that the assumed scaling implies%
\begin{equation*}
\Lambda \left( y\right) =y^{2/\theta }\Lambda \left( 1\right) 
\end{equation*}%
and hence $\Lambda ^{\prime }\left( 1\right) =\frac{2}{\theta }\Lambda
\left( 1\right) $. The rest is obvious.
\end{proof}

\bigskip 

\section{Computational aspects}

We present briefly the \textit{mechanics} of the actual computations, in the
spirit of the Pontryagin maximum principle (e.g. \cite{SS}), the aim being
to find optimal paths which arrive at  $N_{\mathrm{a}}=\left( \mathrm{a},\cdot \right)$, i.e. a given 
"target" manifold. \footnote{We have $\mathrm{a} = \mathrm{y} \in  \mathbb{R}^{l}$, in context of small noise and short time expansions,
and $ \mathrm{a}=1 \in   \mathbb{R}^{l}$, with $l=1$, in the context of tail expansions, corollary \ref{CorTail}.}
This formalism is justified by assuming non-degeneracy of the so-called \textit{%
deterministic Malliavin covariance matrix} $C\left( \mathrm{h}_{0}\right) $,
at each $\mathrm{h}_{0}\in \mathcal{K}_{\mathrm{a}}^{\min }$ ; cf. \cite{TW,
DFJVpartI}. As pointed out earlier, a sufficient condition met in most
("locally elliptic") financial models reads%
\begin{equation*}
\forall \mathrm{h}\in \mathcal{K}_{\mathrm{a}}:\exists t\in \left[ 0,T\right]
:\mathrm{span}\left[ \sigma _{1},\dots ,\sigma _{m}\right] |_{\phi _{t}^{%
\mathrm{h}}}=\mathbb{R}^{d}.
\end{equation*}%
(This is proved in \cite{DFJVpartI}; for more information on $C\left( 
\mathrm{h}_{0}\right) $ see \cite{DFJVpartI} and the references therein; in
particular \cite{Bismut, TW}).

\begin{itemize}
\item \textbf{The Hamiltonian. }Based on the SDE (\ref{SDEXeps}), with
diffusion vector fields $\sigma _{1},\dots ,\sigma _{m}$ and drift vector
field $\sigma _{0}$ (in the $\varepsilon \rightarrow 0$ limit) we define the 
\textit{Hamiltonian}%
\begin{eqnarray*}
\mathcal{H}\left( \mathrm{x},\mathrm{p}\right)  &:&=\left\langle \mathrm{p}%
,\sigma _{0}\left( \mathrm{x}\right) \right\rangle +\frac{1}{2}%
\sum_{i=1}^{m}\left\langle \mathrm{p},\sigma _{i}\left( \mathrm{x}\right)
\right\rangle ^{2} \\
&=&\left\langle \mathrm{p},\sigma _{0}\left( \mathrm{x}\right) \right\rangle
+\frac{1}{2}\left\langle \mathrm{p},\left( \sigma \sigma ^{T}\right) \left( 
\mathrm{x}\right) \mathrm{p}\right\rangle .
\end{eqnarray*}%
Remark the driving Brownian motions  $W^{1},\dots ,W^{m}$ were assumed to be
independent. Many stochastic models, notably in finance, are written in
terms of correlated Brownians, i.e. with a non-trivial correlation matrix $%
\Omega =\left( \omega ^{i,j}:1\leq i,j\leq m\right) $, where $%
d\,\left\langle W^{i},W^{j}\right\rangle _{t}=\omega ^{i,j}dt$. The
Hamiltonian then becomes%
\begin{equation}
\mathcal{H}\left( \mathrm{x},\mathrm{p}\right) =\left\langle \mathrm{p}%
,\sigma _{0}\left( \mathrm{x}\right) \right\rangle +\frac{1}{2}\left\langle 
\mathrm{p},\left( \sigma \Omega \sigma ^{T}\right) \left( \mathrm{x}\right) 
\mathrm{p}\right\rangle .  \label{HamiltonianWithCorrel}
\end{equation}

\item \textbf{The Hamiltonian ODEs.} The following system of ordinary
differential equations,  
\begin{equation}
\left( 
\begin{array}{c}
\mathrm{\dot{x}} \\ 
\mathrm{\dot{p}}%
\end{array}%
\right) =\left( 
\begin{array}{c}
\partial _{\mathrm{p}}\mathcal{H}\left( \mathrm{x}\left( t\right) ,\mathrm{p}%
\left( t\right) \right)  \\ 
-\partial _{\mathrm{x}}\mathcal{H}\left( \mathrm{x}\left( t\right) ,\mathrm{p%
}\left( t\right) \right) 
\end{array}%
\right) ,  \label{HamiltonianODEs}
\end{equation}%
gives rise to a solution flow, denoted by $\mathrm{H}_{t\leftarrow 0}$, so
that%
\begin{equation*}
\mathrm{H}_{t\leftarrow 0}\left( \mathrm{x}_{0},\mathrm{p}_{0}\right) 
\end{equation*}%
is the unique solution to the above ODE with initial data $\left( \mathrm{x}%
_{0},\mathrm{p}_{0}\right) $. Our standing (regularity) assumption are more
than enough to guarantee uniqueness and local ODE\ existence. As in \cite[%
p.37]{Bismut}, the vector\ field $\left( \partial _{\mathrm{p}}\mathcal{H}%
,-\partial _{\mathrm{x}}\mathcal{H}\right) $ is complete, i.e.one has global
existence. It can be usefult to start the flow backwards with time-$T$
terminal data, say $\left( \mathrm{x}_{T},\mathrm{p}_{T}\right) $; we then
write%
\begin{equation*}
\mathrm{H}_{t\leftarrow T}\left( \mathrm{x}_{T},\mathrm{p}_{T}\right) 
\end{equation*}%
for the unique solution to (\ref{HamiltonianODEs}) with given time-$T$
terminal data. Of course,%
\begin{equation*}
\mathrm{H}_{t\leftarrow T}\left( \mathrm{H}_{T\leftarrow 0}\left( \mathrm{x}%
_{0},\mathrm{p}_{0}\right) \right) =\mathrm{H}_{t\leftarrow 0}\left( \mathrm{%
x}_{0},\mathrm{p}_{0}\right) .
\end{equation*}

\item \textbf{\ Solving the Hamiltonian ODEs as boundary value problem. }%
As before, $N_{\mathrm{a}}=\left( \mathrm{a},\cdot \right)$ is the given "target" manifold; the
analysis laid out in \cite{DFJVpartI} requires in a first step to solve the Hamiltonian ODEs (\ref%
{HamiltonianODEs}) with mixed initial -, terminal - and transversality
conditions,%
\begin{eqnarray}
\mathrm{x}\left( 0\right)  &=&\mathrm{x}_{0}\in \mathbb{R}^{d},  \notag \\
\mathrm{x}\left( T\right)  &=&\left( \mathrm{a},\cdot \right) \in \mathbb{R}%
^{l}\mathbb{\oplus R}^{d-l},  \notag \\
\mathrm{p}\left( T\right)  &=&\left( \cdot ,0\right) \in \mathbb{R}^{l}%
\mathbb{\oplus R}^{d-l}.  \label{ITTcond}
\end{eqnarray}%
Note that this is a $2d$-dimensional system of ordinary differential
equations, subject to $d+l+\left( d-l\right) =2d$ conditions. In general,
boundary problems for such ODEs may have more than one, exactly one or no
solution. In the present setting, there will always be one or more than one
solution. After all, we know \cite{DFJVpartI} that there exists at least one
minimizing control $\mathrm{h}_{0}$ and can be reconstructed via the
solution of the Hamiltonian ODEs, as explained in the following step.

\item  \textbf{Finding the minimizing controls. }The Hamiltonian ODEs, as
boundary value problem,  are effectively first order conditions (for
minimality) and thus yield \textit{candidates} for the minimizing control $%
\mathrm{h}_{0}=\mathrm{h}_{0}\left( \cdot \right) $, given by%
\begin{equation}
\mathrm{\dot{h}}_{0}=\left( 
\begin{array}{c}
\,\left\langle \sigma _{1}\left( \mathrm{x}\left( \cdot \right) \right) ,%
\mathrm{p}\left( \cdot \right) \right\rangle  \\ 
\dots  \\ 
\,\,\left\langle \sigma _{m}\left( \mathrm{x}\left( \cdot \right) \right) ,%
\mathrm{p}\left( \cdot \right) \right\rangle 
\end{array}%
\right) .  \label{Formula_h0}
\end{equation}%
Each such candidate is indeed admissible in the sense $\mathrm{h}_{0}\in 
\mathcal{K}_{\mathrm{a}}$ but may fail to be a minimizer. We thus compute
the energy $\left\Vert \mathrm{h}_{0}\right\Vert _{H}^{2}$ for each
candidate and identify those (" $\mathrm{h}_{0}\in \mathcal{K}_{\mathrm{a}%
}^{\min }$") with minimal energy. The procedure via Hamiltonian flows also
yields a unique $\mathrm{p}_{0}=\mathrm{p}_{0}\left( \mathrm{h}_{0}\right) $.

\item \textbf{Checking non-focality. }By definition \cite{DFJVpartI}, $%
\mathrm{x}_{0}$ is \textbf{non-focal} for $N=\left( \mathrm{a},\cdot \right) 
$ along $\mathrm{h}_{0}\in \mathcal{K}_{\mathrm{a}}^{\min }$ in the sense
that, with $((\mathrm{a},\cdot),(\cdot,0)) \ni \left( \mathrm{x}_{T}, \mathrm{p}_{T}\right) :=\mathrm{H}%
_{T\leftarrow 0}\left( \mathrm{x}_{0}\mathrm{,p}_{0}\left( \mathrm{h}%
_{0}\right) \right) \in \mathcal{T}^{\ast }\mathbb{R}^{d}$, 
\begin{equation*}
\partial _{\left( \mathfrak{z},\mathfrak{q}\right) }|_{\left( \mathfrak{z},%
\mathfrak{q}\right) \mathfrak{=}\left( 0,0\right) }\pi \mathrm{H}%
_{0\leftarrow T}\left( \mathrm{x}_{T}+\left( 
\begin{array}{c}
0 \\ 
\mathfrak{z}%
\end{array}%
\right) ,\mathrm{p}_{T}+\left( \mathfrak{q},0\right) \right) 
\end{equation*}%
is non-degenerate (as $d\times d$ matrix; here we think of $\left( \mathfrak{%
z},\mathfrak{q}\right) \in \mathbb{R}^{d-l}\times \mathbb{R}^{l}\cong 
\mathbb{R}^{d}$ and recall that $\pi $ denotes the projection from $\mathcal{%
T}^{\ast }\mathbb{R}^{d}$ onto $\mathbb{R}^{d}$; in coordinates $\pi \left( 
\mathrm{x},\mathrm{p}\right) =\mathrm{x}$). Note that in the point-point
setting, $\mathrm{x}_{T}=\mathrm{x}$ for fixed $\mathrm{x}$, only perturbations of the
arrival "velocity" $\mathrm{p}_{T}$ - without any restrictions of transversality type - are considered. Non-degeneracy of the resulting
map should then be called \textbf{non-conjugacy (}between two points; here: $%
\mathrm{x}_{T}$ and $\mathrm{x}_{0}$). In the Riemannian setting this is
consistent with the usual meaning of non-conjugacy; after identifying
tangent- and cotangent-space $\partial _{\mathfrak{q}}|_{\mathfrak{q=}0}\pi 
\mathrm{H}_{0\leftarrow T}$ is precisely the differential of the exponential
map.

\item   \textbf{The explicit marginal density expansion. }We then have 
\begin{equation*}
\text{ }f^{\varepsilon }\left( \mathrm{a},T\right) =e^{-c_{1}/\varepsilon
^{2}}e^{c_{2}/\varepsilon }\varepsilon ^{-l}\left( c_{0}+O\left( \varepsilon
\right) \right) \text{ as }\varepsilon \downarrow 0.
\end{equation*}%
with $c_{1}=\Lambda \left( \mathrm{a}\right) $. The second-order exponential
constant $c_{2}$ then requires the solution of a finitely many ($\#\mathcal{%
K}_{\mathrm{a}}^{\min }<\infty $) auxilary ODEs, cf. theorem \ref%
{thm:MainThm}. At last we set $\mathrm{a} = \mathrm{y} \in  \mathbb{R}^{l}$, in context of small noise and short time expansions,
and $ \mathrm{a}=1 \in  \mathbb{R}^{l}$, with $l=1$ for tail expansions (in this case, the $y$-dependence here is hidden
in $\varepsilon$).
\end{itemize}

\section{Application to asset price models}

\subsection{Black-Scholes}

The Black-Scholes ($\mathrm{BS}$) model, written in terms of  $\log $-price
is an example where the above theorem is applicable with $\theta =1$.
Indeed, $Y:=\log S$ satisfies, with fixed Black-Scholes volatility $\sigma >0
$ 
\begin{equation*}
dY_{t}=-\frac{\sigma ^{2}}{2}dt+\sigma dW_{t},\,\,\,\,Y_{0}=y_{0}=\log S_{0}.
\end{equation*}%
Of course, $Y_{t}\sim N\left( y_{0}-\sigma ^{2}t/2,\sigma ^{2}t\right) $ and
the explicit Gaussian density%
\begin{equation*}
f_{\mathrm{BS}}\left( t,y\right) =\frac{1}{\sqrt{2\pi \sigma ^{2}t}}\exp
\left\{ -\frac{\left( y-\left( y_{0}-\sigma ^{2}t/2\right) \right) ^{2}}{%
2\sigma ^{2}t}\right\} 
\end{equation*}%
immediately yields short time resp. tail expansions,%
\begin{eqnarray}
f_{\mathrm{BS}}\left( t,y\right)  &\sim &\text{(const)\thinspace }%
t^{-1/2}\exp \left( -\frac{\left( \frac{y-y_{0}}{\sigma }\right) ^{2}}{2t}%
\right) \text{ as }t\downarrow 0\text{; any }y\in \mathbb{R}
\label{BSshorttime} \\
f_{\mathrm{BS}}\left( T,y\right)  &\sim &\text{(const)\thinspace }\exp
\left( -\frac{1}{2\sigma ^{2}T}y^{2}\right) \exp \left( \frac{y_{0}-\sigma
^{2}T/2}{2\sigma ^{2}T}y\right) \text{as }y\rightarrow \infty \text{; any }%
T>0\text{.}  \label{BStail}
\end{eqnarray}%
We derive now both expansions from general theory, i.e. with aid of
corollary \ref{CorShortTime} resp \ref{CorTail}. The short time limit
corresponds to a flat Riemannian situation, in particular the cutlocus is
empty, which is enough to guarantee (\textrm{ND}); the remaining
computations to derive (\ref{BSshorttime}) from corollary \ref{CorShortTime}
are left to the reader and we focus on the (more interesting) case of tail
asymptotics. Corollary \ref{CorTail} applies with $\theta =1$, and
(rescaled) starting point $\varepsilon y_{0}\rightarrow 0$. Condition (%
\textrm{ND}) needs to be checked; the relevant Hamiltonian is 
\begin{equation*}
\mathcal{H}\left( y,p\right) =\frac{\sigma ^{2}p^{2}}{2},\qquad \text{for
all }(y,p)\in \mathbb{R}^{2}
\end{equation*}%
and the Hamiltonian ODEs are 
\begin{equation*}
\dot{y}_{t}=\sigma ^{2}p_{t},\,\,\,\,\,\,\dot{p}_{t}=0,
\end{equation*}%
with boundary conditions $y_{0}=0$ and $y_{T}=1=:a$. Since $p_{t}$ is
constant, we obtain $p_{t}\equiv p_{0}=\frac{a}{\sigma ^{2}T}$, and $y_{t}=a%
\frac{t}{T}$. In particular, $\left. \partial _{p}y_{T}\right\vert
_{p_{0}}=\sigma ^{2}T>0$, and hence invertible, for $T,\sigma >0$. En
passant, we also deduce the optimal control $h_{0}(t)=\sigma p_{0}$, and get
the correct leading order factor 
\begin{equation*}
c_{1}:=\frac{1}{2}\left\Vert h_{0}\right\Vert ^{2}=\frac{1}{2}%
\int_{0}^{T}h_{0}(t)^{2}dt=\frac{1}{2\sigma ^{2}T}.
\end{equation*}%
With the hint $\hat{Y}_{t}=y_{0}+\left( -\frac{\sigma ^{2}}{2}\right) t$ we
leave it to the reader to verify that $c_{2}=\left( y_{0}-\sigma
^{2}T/2\right) /\left( 2\sigma ^{2}T\right) $. Frequently, one chooses $%
y_{0}=0$ in this context (which amounts to normalize spot price to unit).

\subsection{The Stein-Stein model\label{SteinSteinSection}}

For given parameters, $a\geq 0,b<0,c>0,\sigma _{0}\geq 0,\rho
=d\,\left\langle W^{1},W^{2}\right\rangle /dt$, the Stein--Stein model
expresses log-price $Y$, under the forward measure, via 
\begin{eqnarray}
dY &=&-\frac{1}{2}Z^{2}dt+ZdW^{1},\,\,Y\left( 0\right) =y_{0}=0
\label{SteinSteinSDE} \\
dZ &=&\left( a+bZ\right) dt+cdW^{2},\,\,\,Z\left( 0\right) =\sigma _{0}>0. 
\notag
\end{eqnarray}%
We will be interested in the behaviour, and in particular the
tail-behaviour, of the probability density function of $Y_{T}$. In fact,
there is no loss of generality to consider $T=1$. Applying Brownian scaling,
it is a straight-forward computation to see that the pair $\left( \tilde{Y},%
\tilde{Z}\right) $ given by 
\begin{equation*}
\tilde{Y}\left( t\right) :=Y\left( tT\right) ,\,\,\,\tilde{Z}\left( t\right)
:=Z\left( tT\right) T^{1/2}
\end{equation*}%
satisfies the same parametric SDE form as Stein-Stein, but with the
following parameter substitutions%
\begin{equation*}
a\leftarrow \tilde{a}\equiv aT^{3/2},b\leftarrow \tilde{b}\equiv
bT,c\leftarrow \tilde{c}\equiv cT,\sigma _{0}\leftarrow \tilde{\sigma}%
_{0}\equiv \sigma _{0}T^{1/2}.
\end{equation*}%
In particular then, $Y_{T}=Y_{T}\left( a,b,c,\sigma _{0},\rho \right) $ has
the same law as $Y_{1}\left( \tilde{a},\tilde{b},\tilde{c},\tilde{\sigma}%
_{0},\rho \right) $.

\subsubsection{The case of zero-correlation}

For the moment, we shall follow \cite{GuSt} in assuming the Brownians to be
uncorrelated,%
\begin{equation*}
d\left\langle W^{1},W^{2}\right\rangle _{t}=\rho dt\text{ with }\rho =0.
\end{equation*}%
Recall their main result, a density expansion for $Y_{T}$ of the form 
\begin{equation}
\left( \ast \right) :f\left( y\right)
=e^{-c_{1}y}e^{c_{2}y^{1/2}}y^{-1/2}\left( c_{3}+O\left( y^{-1/2}\right)
\right) \text{ as }y\rightarrow \infty .  \label{GSexpansion}
\end{equation}

\textbf{Scaling: }Setting 
\begin{equation*}
Y_{\varepsilon }:=\varepsilon ^{2}Y,\,Z_{\varepsilon }:=\varepsilon Z
\end{equation*}%
yields the small noise problem 
\begin{eqnarray}
dY_{\varepsilon } &=&-\frac{1}{2}Z_{\varepsilon }^{2}dt+Z_{\varepsilon
}\varepsilon dW^{1},\,\,Y_{\varepsilon }\left( 0\right) =0=:y_{0}\text{ }%
\forall \varepsilon >0  \label{SNP} \\
dZ_{\varepsilon } &=&\left( a\varepsilon +bZ_{\varepsilon }\right)
dt+c\varepsilon dW^{2},\,\,\,Z_{\varepsilon }\left( 0\right) =\varepsilon
\sigma _{0}\rightarrow 0=:z_{0}\text{ as }\varepsilon \downarrow 0.  \notag
\end{eqnarray}%
Our corollary \ref{CorTail}, assuming its application to be justified, then
gives the correct expansion (\ref{GSexpansion}), namely%
\begin{equation*}
f\left( y\right) =e^{-c_{1}\,y}e^{c_{2}\,y^{1/2}}y^{-1/2}\left(
c_{3}+O\left( y^{-1/2}\right) \right) ,
\end{equation*}%
and also identifies the constants $c_{1}=\Lambda \left( 1\right) $, $c_{2}=%
\hat{Y}_{T}\Lambda ^{\prime }\left( 1\right) $. (The leading order constant $%
c_{1}$is in agreement with both \cite{GuSt} and \cite[p40]{Deuschel}.)

\begin{remark}
\bigskip Corollary \ref{CorTail} relies on an application of theorem \ref%
{thm:MainThm} to (\ref{SNP}); let us note straight away that the
coefficients here are smooth but unbounded. With a view towards the earlier
remark on localization, and in particular (\ref{Assloc}), we note here that,
due to the particular structure of the\ SDE, it suffices to localize such as
to make $\sigma $ bounded; e.g. by stopping it upon leaving a big ball of
radius $R$. This amounts to, cf. (\ref{Assloc}), to shows that%
\begin{equation*}
\lim_{R\rightarrow \infty }\lim \sup_{\varepsilon \rightarrow 0}\varepsilon
^{2}\log \mathbb{P}\left[ \left\vert \sigma _{\varepsilon }\right\vert
_{\infty ;\left[ 0,T\right] }\geq R\right] =-\infty .
\end{equation*}%
But since $\mathbb{P}\left[ \left\vert \sigma _{\varepsilon }\right\vert
_{\infty ;\left[ 0,T\right] }\geq R\right] =\mathbb{P}\left[ \left\vert
\sigma \right\vert _{\infty ;\left[ 0,T\right] }\geq R/\varepsilon \right] $
and $\sigma $ is a Gaussian process, this is an immediate consequence of
Fernique's estimate.
\end{remark}

We postpone the justification that we may indeed apply corollary \ref%
{CorTail} (which involves an analysis of the Hamiltonian ODEs) and proceed
in showing how further qualitative information about the expansion can be
obtained without much computations.\newline
\textbf{Some information on }$c_{1}$\textbf{:} According to theorem \ref%
{thm:MainThm}, 
\begin{equation*}
c_{1}:=\Lambda \left( 1\right) =\inf \left\{ \frac{1}{2}\Vert \mathrm{h}%
\Vert _{H}^{2}:\phi _{0}^{\mathrm{h}}=\left( 0,0\right) ,\phi _{T}^{\mathrm{h%
}}\in \left( 1,\cdot \right) \right\} 
\end{equation*}%
where $d\phi _{t}^{\mathrm{h},1}=-\frac{1}{2}\left\vert \phi _{t}^{\mathrm{h}%
,2}\right\vert ^{2}dt+\phi _{t}^{\mathrm{h},2}dh^{1},d\phi _{t}^{\mathrm{h}%
,2}=b\phi _{t}^{\mathrm{h},2}dt+cdh^{2}.$ If then follows \textit{a priori}
that%
\begin{equation*}
c_{1}=c_{1}\left( b,c;T\right) \text{ but not on }a,\sigma _{0}.
\end{equation*}%
The same is true for $\mathrm{h}_{0}=:\mathrm{h}^{\ast }=:\left( h^{\ast
,1},h^{\ast ,2}\right) $ and $\phi ^{\ast }:=\phi ^{\mathrm{h}_{0}}$ of
course.\newline
\textbf{Some information on }$c_{2}$\textbf{: }First, $\Lambda ^{\prime
}\left( 1\right) =c_{1}$ also only depends on the parameters $b,c,T$ (but
not on $a,\sigma _{0}$). It remains to analyze the factor $\hat{Y}_{T}$
where $\left( \hat{Y}_{t},\hat{Z}_{t}:t\geq 0\right) $ solves the ODE%
\begin{eqnarray*}
d\hat{Y}_{t} &=&\left( -\phi _{t}^{\ast ,2}+h_{t}^{\ast ,1}\right) \hat{Z}%
_{t}dt,\,\,\,\,\hat{Y}_{0}=0 \\
d\hat{Z}_{t} &=&b\hat{Z}_{t}dt+adt,\,\,\,\,\hat{Z}_{0}=\sigma _{0}.
\end{eqnarray*}%
Since $\hat{Z}_{t}=\sigma _{0}e^{bt}+a\int_{0}^{t}e^{b\left( t-s\right) }ds$
it follows that $\hat{Z}_{T}$ is linear in $\sigma _{0},a$ with coefficients
depending on $b$ and $T$. Furthermore, noting that%
\begin{equation*}
\hat{Y}_{T}=\int_{0}^{T}\left( -\phi _{t}^{\ast ,2}+h_{t}^{\ast ,1}\right) 
\hat{Z}_{t}dt
\end{equation*}%
a similar statement is true for $\hat{Y}_{T}$ and then $c_{2}=\Lambda
^{\prime }\left( 1\right) \times \hat{Y}_{T}^{1}$. Namely, for constants $%
C_{i}=C_{i}\left( b,c;T\right) $%
\begin{equation*}
c_{2}=C_{1}\left( b,c;T\right) \sigma _{0}+C_{2}\left( b,c;T\right) a.
\end{equation*}%
It is interesting to compare this with the Heston result \cite{FGGS} where
the constant $c_{2}$ also depends linearly on spot-vol $\sigma _{0}=\sqrt{%
v_{0}}$.

\paragraph{Solving the Hamiltonian ODEs and computing $c_{1}$}

After replacing $\varepsilon dW$ by a control $d\mathrm{h}$, and taking $%
\varepsilon \downarrow 0$ elsewhere in (\ref{SNP}), we have to consider the
controlled ordinary differential equation 
\begin{eqnarray}
dy &=&-\frac{1}{2}z^{2}dt+zdh^{1},\,\,y_{0}=0  \label{SSeps0} \\
dz &=&bzdt+cdh^{2},\,\,\,z_{0}=0,  \notag
\end{eqnarray}%
minimizing the energy, $\frac{1}{2}\int_{0}^{T}\left\vert \mathrm{\dot{h}}%
_{t}\right\vert ^{2}dt$ subject to $y_{T}=\mathrm{a}\equiv 1>0$.

According to general theory, we now write out the Hamiltonian associated to (%
\ref{SSeps0}), 
\begin{eqnarray}
&&\mathcal{H}\left( \left( 
\begin{array}{c}
y \\ 
z%
\end{array}%
\right) ;\left( p,q\right) \right)   \label{SSHam} \\
&=&%
\begin{pmatrix}
-\frac{1}{2}z^{2} \\ 
bz%
\end{pmatrix}%
\cdot \left( 
\begin{array}{c}
p \\ 
q%
\end{array}%
\right) +\frac{1}{2}\left\vert 
\begin{pmatrix}
z \\ 
0%
\end{pmatrix}%
\cdot \left( 
\begin{array}{c}
p \\ 
q%
\end{array}%
\right) \right\vert ^{2}+\frac{1}{2}\left\vert 
\begin{pmatrix}
0 \\ 
c%
\end{pmatrix}%
\cdot \left( 
\begin{array}{c}
p \\ 
q%
\end{array}%
\right) \right\vert ^{2}  \notag \\
&=&-\frac{1}{2}z^{2}p+bzq+\frac{1}{2}\left( z^{2}p^{2}+c^{2}q^{2}\right) . 
\notag
\end{eqnarray}

The Hamiltonian ODEs then become 
\begin{eqnarray}
\left( 
\begin{array}{c}
\dot{y}_{t} \\ 
\dot{z}_{t}%
\end{array}%
\right) &=&%
\begin{pmatrix}
z_{t}^{2}\left( p_{t}-\frac{1}{2}\right) \\ 
bz_{t}+c^{2}q_{t}%
\end{pmatrix}
\label{HamiltonODEs} \\
\quad \left( 
\begin{array}{c}
\dot{p}_{t} \\ 
\dot{q}_{t}%
\end{array}%
\right) &=&%
\begin{pmatrix}
0 \\ 
p_{t}z_{t}\left( 1-p_{t}\right) -bq_{t}%
\end{pmatrix}%
.  \notag
\end{eqnarray}%
Trivially, $p_{t}\equiv p_{0}$ which we shall denote by $p$ from here on. As
it turns out there is a simple expression for the energy. Although we shall
ultimately take $\mathrm{a}\equiv 1$ it is convenient to carry out the
following analysis for general $\mathrm{a}>0$.

\begin{lemma}
\label{LemmaEnergySS_PosP}For any $\mathrm{h}_{0}\in \mathcal{K}_{\mathrm{a}%
}^{\min }$, and in fact any $\mathrm{h}_{0}$ given by (\ref{Formula_h0}),
i.e.%
\begin{equation}
\mathrm{\dot{h}}_{0}\left( t\right) =\left( 
\begin{array}{c}
pz_{t} \\ 
q_{t}c%
\end{array}%
\right)  \label{SSh0dotExplicit}
\end{equation}%
where $\left( y,z;p,q\right) $ satisfies (\ref{HamiltonODEs}), subject to
boundary conditions $\left( y_{0},z_{0}\right) =(0,0)$ and $y_{T}=\mathrm{a}%
,q_{T}=0$, we have%
\begin{equation*}
\Lambda \left( \mathrm{a}\right) =\frac{1}{2}\int_{0}^{T}\left\vert \mathrm{%
\dot{h}}_{0}\left( t\right) \right\vert ^{2}dt=p\mathrm{a}.
\end{equation*}%
In particular, we see that 
\begin{equation*}
p\geq 0\text{.}
\end{equation*}
\end{lemma}

\begin{remark}
In fact, linearity in $\mathrm{a}$ of (\ref{elemyetted2}) also follows
immediately from the fact that the Stein-Stein model satisfies $\theta $%
-scaling with $\theta =2$ in the sense of corollary \ref{CorTail}. Indeed,
it was seen in the proof of that corollary that the rate function $\Lambda
\left( \mathrm{a}\right) $ scales like $\mathrm{a}^{2/\theta }=\mathrm{a}$.
This already implies that $p$ does not depend on $\mathrm{a}$. This is also
consistent with the principle $\partial _{\mathrm{a}}\Lambda \left( \mathrm{a%
}\right) =p_{T}$ pointed out in \cite{DFJVpartI}.
\end{remark}

\begin{proof}
We give an elegant argument based on the Hamiltonian ODEs. The idea is to
express $\left\vert \mathrm{\dot{h}}_{0}\left( t\right) \right\vert ^{2}$as
a time-derivative which then allows for immediate integration over $t\in %
\left[ 0,T\right] $. Indeed,%
\begin{eqnarray*}
\left\vert \mathrm{\dot{h}}_{0}\left( t\right) \right\vert ^{2}
&=&p^{2}z_{t}^{2}+c^{2}q_{t}^{2} \\
&=&p^{2}z_{t}^{2}+\partial _{t}\left( z_{t}q_{t}\right) -z_{t}^{2}\left(
p^{2}-p\right) \\
&=&2pz_{t}^{2}\left( p-1/2\right) +\partial _{t}\left( z_{t}q_{t}\right) \\
&=&2p\dot{y}_{t}+\partial _{t}\left( z_{t}q_{t}\right)
\end{eqnarray*}%
where we used the ODEs for $z,q$ as given in (\ref{HamiltonODEs}). It
follows that%
\begin{equation*}
\int_{0}^{T}\left\vert \mathrm{\dot{h}}_{0}\left( t\right) \right\vert
^{2}dt=2p\left( y_{T}-y_{0}\right) +\left( z_{T}q_{T}-z_{0}q_{0}\right)
\end{equation*}%
and we conclude with the initial/terminal/transversality conditions $%
y_{0}=z_{0}=0,\,y_{T}=\mathrm{a}$ and $q_{T}=0$.
\end{proof}

\begin{lemma}[Partial Hamiltonian Flow]
Consider (\ref{HamiltonODEs}) as initial value problem, with initial data $%
\left( y_{0},z_{0}\right) =(0,0)$ and $(p,q_{0})$. Assume\footnote{%
All explicit solutions given in (\ref{HSolutionFlow}) are even functions of $%
\chi _{p_{0}}$ and have a removable singularity for $\chi _{p_{0}}=0$. By
convention we shall always assume $\chi _{p_{0}}\geq 0$ although the sign of 
$\chi _{p_{0}}$ does not matter.}%
\begin{equation}
\chi _{p}^{2}:=c^{2}p\left( p-1\right) -b^{2}\geq 0.  \label{chi2p}
\end{equation}%
Then the explicit solution is given by%
\begin{align}
y_{t}& =\frac{q_{0}^{2}c^{4}\left( 2p_{0}-1\right) }{8\chi _{p}^{3}}\left(
2\chi _{p}t-\sin \left( 2\chi _{p}t\right) \right) ,  \label{HSolutionFlow}
\\
z_{t}& =\frac{q_{0}c^{2}}{\chi _{p}}\sin \left( \chi _{p}t\right) ,  \notag
\\
p_{t}& \equiv p,  \notag \\
q_{t}& =q_{0}\left( \cos \left( \chi _{p}t\right) -\frac{b}{\chi _{p}}\sin
\left( \chi _{p}t\right) \right) .  \notag
\end{align}
\end{lemma}

\begin{remark}
The given solutions remain valid when $\chi _{p}^{2}<0$; it suffices to
consider $\chi _{p}$ as purely imaginary; then, if desired, rewrite as $\cos
\left( \chi _{p}t\right) =\cosh \left( \left\vert \chi _{p}\right\vert
t\right) $ etc. Below, we shall solve (\ref{HamiltonODEs}) as boundary value
problem, subject to $\left( y_{0},z_{0}\right) =(0,0),$ $y_{T}=\mathrm{a}>0$
and $q_{T}=0$; we shall see then that (\ref{chi2p}) is always satisfied and
in fact $\chi _{p}^{2}>0$.
\end{remark}

\begin{proof}
Let us first remark that the path $\left( p_{t}\right) _{t\geq 0}$ is
constant, $p_{t}=p$ for all $t\in \lbrack 0,T]$. From the Hamiltonian ODEs,
the couple $\left( z_{t},q_{t}\right) _{t\geq 0}$ solves a linear ODE in $%
\mathbb{R}^{2}$, so that the solution must be a linear function of $\left(
z_{0},q_{0}\right) =\left( 0,q_{0}\right) $. Indeed, a simple computation
gives 
\begin{equation*}
q_{t}=q_{0}\left( \cos \left( \chi _{p}t\right) -\frac{b}{\chi _{p}}\sin
\left( \chi _{p}t\right) \right) \qquad \text{and}\qquad z_{t}=\frac{%
q_{0}c^{2}}{\chi _{p}}\sin \left( \chi _{p}t\right) ,
\end{equation*}%
Elementary integration ("$2\int_{0}^{t}\sin ^{2}=t-\cos \sin t$") then gives 
$\left( y_{t}\right) _{t\geq 0}$ by direct integration; indeed 
\begin{equation*}
y_{t}=\left( p-\frac{1}{2}\right) \int_{0}^{t}z_{s}^{2}ds=\frac{%
q_{0}^{2}c^{4}\left( 2p-1\right) }{8\chi _{p}^{3}}\left( 2\chi _{p}t-\sin
\left( 2\chi _{p}t\right) \right) .
\end{equation*}%
This proves the lemma.
\end{proof}

For the next proposition we recall the standing assumptions $T>0,$ $b\leq 0$
(which models mean-reversion) and $\mathrm{a>0}$.

\begin{proposition}
\label{PropSolSSHamODE_BVP}The ensemble of solutions to the Hamilton ODEs as
boundary value problem%
\begin{equation*}
\left( y_{0},z_{0}\right) =(0,0)\text{ and }y_{T}=\mathrm{a},q_{T}=0
\end{equation*}%
with $\mathrm{a}=1>0$ are characterized by inserting, for any $k\in \left\{
1,2,...\right\} $ and any choice of sign in (\ref{q0explicit}) below, 
\begin{eqnarray}
p &=&p_{k}=\frac{1}{2}\left( 1+\sqrt{1+\frac{4b^{2}}{c^{2}}+\frac{4r_{k}^{2}%
}{c^{2}T^{2}}}\right) ,  \label{p0explicit} \\
q_{0,k}^{\pm } &=&\pm \frac{2}{c^{2}}\sqrt{\frac{2r_{k}^{3}\,\mathrm{a}}{%
\left( 2p_{0,k}^{+}-1\right) T^{3}\left( 2r_{k}-\sin \left( 2r_{k}\right)
\right) }}  \label{q0explicit}
\end{eqnarray}%
in (\ref{HSolutionFlow}). Here $\left\{ r_{k}:k=1,2,\dots \right\} $ denotes
the set of (increasing) strictly positive roots to%
\begin{equation*}
r\cos (r)-bT\sin (r)=0\text{.}
\end{equation*}
\end{proposition}

\begin{remark}
As the proof will show, $p$ as given in (\ref{p0explicit}) is the unique
positive root to%
\begin{equation*}
c^{2}p\left( p-1\right) -b^{2}=\left( \frac{r_{0,k}}{T}\right) ^{2};
\end{equation*}%
in particular, assumption (\ref{chi2p}) in the previous lemma is met.
\end{remark}

\begin{proof}
By assumption and (\ref{HSolutionFlow}), 
\begin{equation}
0=q_{T}=q_{0}\left( \cos \left( \chi _{p}T\right) -\frac{b}{\chi _{p}}\sin
\left( \chi _{p}T\right) \right) .  \label{0equq_T}
\end{equation}%
At this stage, $\chi _{p}$ could be a complex number (when $\chi _{p}^{2}<0$%
). Let us note straight away that we must have $q_{0}\neq 0$ for otherwise $%
\left( y_{t}\right) _{t\geq 0}$ - which depends linearly on $q_{0}$ as is
seen explicitly in (\ref{HSolutionFlow}) - would be identically equal to
zero in contradiction with $y_{T}=\mathrm{a}>0$. Let us also note that $\chi
_{p}\neq 0$ for otherwise (\ref{0equq_T}), which has a removable singularity
at $\chi _{p}=0$, leads to the contradiction $0=1-bT.$(Recall $b\leq 0,T>0$%
.) But then $r:=$ $\chi _{p}T$ is a root, i.e. maps to zero, under the map 
\begin{equation}
r\in \mathbb{C}\mapsto r\cos r-bT\sin r=r\left( \cos r-\frac{bT}{r}\sin
r\right) \text{.}  \label{rCmap}
\end{equation}%
A complex analysis lemma \cite[Lemma 4]{GuSt} asserts that this map, provided%
\begin{equation}
-bT\geq 0,  \label{minusBTpos}
\end{equation}%
has only real roots; it follows that $\chi _{p}$ is real and so $\chi
_{p}^{2}\geq 0$; actually $\chi _{p}^{2}>0$, since we already noted that $%
\chi _{p}\neq 0$. Note that (\ref{0equq_T}), and in fact all further
expressions involving $\chi _{p}$, are unchanged upon changing sign of $\chi
_{p}$, we shall agree to take $\chi _{p}>0$ as the positive square-root of $%
\chi _{p}^{2}$. In particular, (\ref{0equq_T}) is equivalent to the
existence of $\chi _{p}>0$ such that%
\begin{equation*}
\chi _{p}T\cos \left( \chi _{p}T\right) -bT\sin \left( \chi _{p}T\right) =0.
\end{equation*}%
It follows that $\chi _{p}T\in \left\{ r_{k}:k=0,1,2,\dots \right\} $, the
set of zeros of (\ref{rCmap}) written in increasing order. We deduce that,
for each $k=0,1,2,\dots $ there is a choice of $p$ arising from%
\begin{equation*}
\chi _{p}^{2}=c^{2}p\left( p-1\right) -b^{2}=\left( \frac{r_{k}}{T}\right)
^{2}.
\end{equation*}%
For each $k$, there is a negative solution, say $p=p_{k}^{-}<0$ which we may
ignore thanks to lemma \ref{LemmaEnergySS_PosP}, and a positive solution,
namely 
\begin{equation*}
p=p_{k}^{+}=\frac{1}{2}\left( 1+\sqrt{1+\frac{4b^{2}}{c^{2}}+\frac{4r_{k}^{2}%
}{c^{2}T^{2}}}\right) >1.
\end{equation*}%
We now exploit $y_{T}=\mathrm{a}$. From the explicit expression of $y_{t}$
given in (\ref{HSolutionFlow}) we get%
\begin{eqnarray*}
\mathrm{a} &=&y_{T}=\frac{q_{0}^{2}c^{4}\left( 2p-1\right) }{8\chi _{p}^{3}}%
\left( 2\chi _{p}T-\sin \left( 2\chi _{p}T\right) \right) \\
&=&\frac{q_{0}^{2}c^{4}\left( 2p-1\right) T^{3}}{8r_{k}^{3}}\left(
2r_{k}-\sin \left( 2r_{k}\right) \right)
\end{eqnarray*}%
and thus%
\begin{equation*}
q_{0}^{2}=\frac{8r_{k}^{3}}{c^{4}\left( 2p-1\right) T^{3}\left( 2r_{k}-\sin
\left( 2r_{k}\right) \right) }\mathrm{a}.
\end{equation*}%
It follows that, for each $k\in \left\{ 1,2,\dots \right\} $, we can take%
\begin{eqnarray*}
\,p &=&p_{k}^{+}=\frac{1}{2}\left( 1+\sqrt{1+\frac{4b^{2}}{c^{2}}+\frac{4}{%
c^{2}}\left( \frac{r_{k}}{T}\right) ^{2}}\right) \\
q_{0} &=&q_{0,k}^{\pm }=\pm \frac{2}{c^{2}}\sqrt{\frac{2r_{k}^{3}\,\mathrm{a}%
}{\left( 2p_{k}^{+}-1\right) T^{3}\left( 2r_{k}-\sin \left( 2r_{k}\right)
\right) }}
\end{eqnarray*}%
and any such choice \bigskip in (\ref{HSolutionFlow}) leads to a solution of
the boundary value problem.
\end{proof}

\bigskip

So far, we have for each $k\in \left\{ 1,2,\dots \right\} $ two choices of $%
\left( p,q_{0}\right) $, depending on the sign in (\ref{q0explicit}) so that
the resulting Hamiltonian ODE solutions, started from $\left(
y_{0},z_{0}\right) =\left( 0,0\right) $ and $\left( p\,,q_{0}\right) $,
describe \textit{all} possible solutions of the boundary value problem given
by the Hamiltonian ODEs with mixed initial/terminal data%
\begin{equation*}
\left( y_{0},z_{0}\right) =(0,0)\text{ and }y_{T}=\mathrm{a},q_{T}=0.
\end{equation*}%
It remains to see which choice (or choices) lead to minimizing controls;
i.e. $\mathrm{h}_{0}\in \mathcal{K}_{\mathrm{a}}^{\min }$. But this is easy
since we know from lemma \ref{LemmaEnergySS_PosP} that, for any $p\in
\left\{ p_{k}^{+}:k=1,2,\dots \right\} $,%
\begin{equation*}
\frac{1}{2}\int_{0}^{T}\left\vert \mathrm{\dot{h}}_{0}\left( t\right)
\right\vert ^{2}dt=p\mathrm{a.}
\end{equation*}%
Since $p_{k}^{+}$ is plainly (strictly) increasing in $k\in \left\{
1,2,\dots \right\} ,$ we see that the energy is minimal if and only if $%
p=p_{1}^{+}$. On the other hand, we are left with two choices for $q_{0}$,
namely $q_{0,1}^{+}$ and $q_{0,1}^{-}$. Using (\ref{SSh0dotExplicit}) we
then see that there are \textit{two} minimizing controls,%
\begin{equation*}
\mathcal{K}_{\mathrm{a}}^{\min }=\left\{ \mathrm{h}_{0}^{+},\mathrm{h}%
_{0}^{-}\right\} ,
\end{equation*}%
given by%
\begin{equation*}
\mathrm{\dot{h}}_{0}^{\pm }\left( t\right) =\left( 
\begin{array}{c}
p\frac{q_{0}c^{2}}{\chi _{p}}\sin \left( \chi _{p}t\right) \\ 
cq_{0}\left( \cos \left( \chi _{p}t\right) -\frac{b}{\chi _{p}}\sin \left(
\chi _{p}t\right) \right)%
\end{array}%
\right) \text{ with }\left( p,q_{0}\right) \leftarrow \left(
p_{1}^{+},q_{0,1}^{+}\right) \text{ resp. }\left(
p_{1}^{+},q_{0,1}^{-}\right) .
\end{equation*}%
Of course, $\mathrm{h}_{0}^{\pm }$ stands for $\mathrm{h}_{0}^{+}$ resp. $%
\mathrm{h}_{0}^{-}$ depending on the chosen substitution above. In $\left(
y,z\right) $-coordinates, note that both $\mathrm{h}_{0}^{+}$ and $\mathrm{h}%
_{0}^{-}$ have identical $y$-components; their $z$-components only differ by
a flipped sign due to $q_{0,1}^{-}=-q_{0,1}^{+}$. (This reflects a
fundamental symmetry in our problem which is in fact invariant under $\left(
y,z\right) \mapsto \left( y,-z\right) $). We summarize our finds in stating
that 
\begin{equation}
\Lambda \left( \mathrm{a}\right) =\frac{1}{2}\Vert \mathrm{h}_{0}^{+}\Vert
_{H}^{2}=\frac{1}{2}\Vert \mathrm{h}_{0}^{-}\Vert _{H}^{2}=p_{1}^{+}\mathrm{a%
}  \label{elemyetted2}
\end{equation}%
and upon taking $\mathrm{a}=1$ we have computed the leading order constant%
\begin{equation*}
c_{1}=\Lambda \left( 1\right) =p_{1}^{+}=\frac{1}{2}\left( 1+\sqrt{1+\frac{%
4b^{2}}{c^{2}}+\frac{4}{c^{2}}\left( \frac{r_{1}}{T}\right) ^{2}}\right)
\end{equation*}%
where we recall that $r_{1}$ is the first strictly positive root of the
equation $r\cos (r)-bT\sin (r)=0$.

\paragraph{Computing $c_{2}$}

According to general theory, cf. equation (\ref{eq:SDEYHat}), we need to
compute certain ODEs for each minimizer, $\mathrm{h}_{0}^{+}=(h_{0,\cdot
}^{+,1},h_{0,\cdot }^{+,2})$ resp. $\mathrm{h}_{0}^{-}=(h_{0,\cdot
}^{-,1},h_{0,\cdot }^{-,2})$, exhibited in the previous section. For ease of
notation we shall write $\left( p,q_{0}^{\pm }\right) $ instead of $\left(
p_{1}^{+},q_{0,1}^{+}\right) $ resp.$\left( p_{1}^{+},q_{0,1}^{-}\right) $
in this section. Related to equation (\ref{SNP}) we then have to consider
the following ODE along $\mathrm{h}_{0}^{+}$ (and then along $\mathrm{h}%
_{0}^{-}$)%
\begin{eqnarray*}
\frac{d}{dt}\left( 
\begin{array}{c}
\hat{Y}_{t} \\ 
\hat{Z}_{t}^{2}%
\end{array}%
\right) &=&\left\{ \left( 
\begin{array}{cc}
0 & -z_{t}^{+} \\ 
0 & b%
\end{array}%
\right) +\left( 
\begin{array}{cc}
0 & 1 \\ 
0 & 0%
\end{array}%
\right) \dot{h}_{0,t}^{+,1}\right\} \left( 
\begin{array}{c}
\hat{Y}_{t} \\ 
\hat{Z}_{t}^{2}%
\end{array}%
\right) +\left( 
\begin{array}{c}
0 \\ 
a%
\end{array}%
\right) \\
&=&\left( 
\begin{array}{cc}
0 & \left( p-1\right) z_{t}^{+} \\ 
0 & 0%
\end{array}%
\right) \left( 
\begin{array}{c}
\hat{Y}_{t} \\ 
\hat{Z}_{t}^{2}%
\end{array}%
\right) +\left( 
\begin{array}{c}
0 \\ 
a%
\end{array}%
\right) \\
\,\,\text{with }\left( 
\begin{array}{c}
\hat{Y}_{0} \\ 
\hat{Z}_{0}^{2}%
\end{array}%
\right) &=&\left( 
\begin{array}{c}
0 \\ 
\sigma _{0}%
\end{array}%
\right) .
\end{eqnarray*}

Here, we used the fact that $\dot{h}_{0}^{+,1}=pz_{t}^{+}$, $z_{t}^{+}$
indicates the chosen sign of $q_{0,1}$ upon which it depends, cf. (\ref%
{q0explicit}). The ODE\ along $\mathrm{h}_{0}^{-}$for $\hat{Y}=\hat{Y}^{-}$
is similar, with $z_{t}^{+},\dot{h}_{0,t}^{+,1}$ replaced by $%
z_{t}^{-}=-z_{t}^{+},\dot{h}_{0,t}^{-,1}=-\dot{h}_{0,t}^{+,1}$ respectively.
We can solve these ODEs explicitly. In a first step (regardless of the
chosen sign for $z,h_{0}$)%
\begin{equation*}
\hat{Z}_{t}=\left\{ 
\begin{array}{c}
\sigma _{0}e^{bT}+\frac{a}{b}\left( e^{bt}-1\right) \text{ for }b<0 \\ 
\sigma _{0}+at\text{ \ \ \ \ for }b=0\text{ \ \ \ \ \ \ \ \ \ \ \ \ \ \ \ \
\ }%
\end{array}%
\right.
\end{equation*}%
and since%
\begin{equation*}
\hat{Y}_{T}^{\pm }=\left( p-1\right) \int_{0}^{T}z_{t}^{\pm }\hat{Z}_{t}dt
\end{equation*}%
we see that $\hat{Y}_{T}^{-}=-\hat{Y}_{T}^{+}$.In fact, under the (usual)
model parameter assumptions $a>0,\sigma _{0}>0$ we see that $\hat{Z}_{t}>0$.
We then note that%
\begin{equation*}
z_{t}^{\pm }/q_{0}^{\pm }=\frac{c^{2}}{\chi _{p}}\sin \left( \chi
_{p}t\right) \geq 0\text{ for }t\in \left[ 0,T\right] ;
\end{equation*}%
indeed we saw that $\chi _{p}T\in \lbrack \pi /2,\pi )$ which implies $\chi
_{p}t\in \lbrack 0,\pi )$ and hence $\sin \left( \chi _{p}t\right) \geq 0$.
In particular, given that $q_{0}^{+}>0$ and $p>1$we see that $\hat{Y}%
_{T}^{+}>0$ (and then $\hat{Y}_{T}^{-}<0$). It follows that 
\begin{eqnarray}
c_{2} &:&=c_{2}^{+}=\Lambda ^{\prime }\left( 1\right) \times \hat{Y}%
_{T}^{+,1}  \notag \\
&=&p\left( p-1\right) \int_{0}^{T}z_{t}^{+}\hat{Y}_{t}^{2}dt  \label{c2asInt}
\end{eqnarray}%
\ whereas the contribution from $c_{2}^{-}=\Lambda ^{\prime }\left( 1\right)
\times \hat{Y}_{T}^{-,1}$ is exponentially smaller and will not figure in
the expansion.
 In fact, given the explicit
form of $t\mapsto z_{t}^{+}$ resp. $\hat{Y}_{t}^{2}$ in terms of $\sin
\left( .\right) $ and $\exp \left( .\right) $, it is clear that the
integration in (\ref{c2asInt}) can be carried out in closed form. In doing
so, one exploits a cancellation due to%
\begin{equation*}
-\chi _{p}\cos \left( \chi _{p}T\right) +b\sin \left( \chi _{p}T\right) =0
\end{equation*}%
and also the equality $\chi _{p}^{2}+b^{2}=c^{2}p\left( p-1\right) $, one is
led to 
\begin{equation*}
c_{2}=q_{0}^{+}\left\{ \sigma _{0}+a\frac{\tan \left( \chi _{p}T/2\right) }{%
\chi _{p}}\right\} .
\end{equation*}%
It is possible, of course, to substitute the explicitly known quantities $%
q_{0}^{+},\chi _{p}$ but this does not yield additional insight.

\subsubsection{The case of non-zero correlation\label{SecSSnonzerocorl}}

We consider again the SDE (\ref{SteinSteinSDE}) with diffusion matrix%
\begin{equation*}
\sigma =\left( \sigma _{1},\sigma _{2}\right) =\left( 
\begin{array}{cc}
z & 0 \\ 
0 & c%
\end{array}%
\right)
\end{equation*}%
but now allow for correlation $\rho $ between $W^{1},W^{2}$; we thus have
the non-trivial correlation matrix 
\begin{equation*}
\Omega =\left( 
\begin{array}{cc}
1 & \rho \\ 
\rho & 1%
\end{array}%
\right) \implies \sigma \Omega \sigma ^{T}=\left( 
\begin{array}{cc}
z^{2} & \rho cz \\ 
\rho cz & c^{2}%
\end{array}%
\right) .
\end{equation*}%
In view of financial applications \cite{Ga06} it makes sense to focus on the
case $\rho \in (-1,0]$. This will also prove convenient in our analysis
below, although there is no doubt that the case $\rho >0$, less interesting
in practice, could also be handled within the present framework.

The Hamiltonian becomes, cf. (\ref{HamiltonianWithCorrel}), 
\begin{eqnarray*}
\mathcal{H}\left( \left( 
\begin{array}{c}
y \\ 
z%
\end{array}%
\right) ;\left( p,q\right) \right)
&=&-\frac{1}{2}z^{2}p+bzq+\frac{1}{2}\left(
z^{2}p^{2}+c^{2}q^{2}\right) +\rho czpq \\
&=&-\frac{1}{2}z^{2}p+\tilde{b}zq+\frac{1}{2}\left(
z^{2}p^{2}+c^{2}q^{2}\right)
\end{eqnarray*}%
with%
\begin{equation*}
\tilde{b}:=\tilde{b}_{p}:=b+\rho cp
\end{equation*}%
Noting $\partial _{\left( y,z\right) }\tilde{b}=\left( 0,0\right) ^{\prime
},\partial _{\left( p,q\right) }\tilde{b}=\left( \rho c,0\right) ^{\prime }$%
. The Hamiltonian equations for $\dot{z},\dot{p},\dot{q},$ are thus
identical as in the uncorrelated case, one just has to replace $b$ by $%
\tilde{b}.$ (In particular, $p_{t}$ is again seen to be constant and we
denote its value by $p$.) The Hamiltonian equation for $\dot{y}=\partial _{p}%
\mathcal{H}$ has, in comparison to the uncorrelated case, an additional
term, namely $\left( \partial _{p}\tilde{b}\right) z_{t}q_{t}=\rho
cz_{t}q_{t}$. In summary, the Hamiltonian ODEs are 
\begin{eqnarray*}
\left( 
\begin{array}{c}
\dot{y}_{t} \\ 
\dot{z}_{t}%
\end{array}%
\right) &=&%
\begin{pmatrix}
z_{t}^{2}\left( p_{t}-\frac{1}{2}\right) +\rho cz_{t}q_{t} \\ 
\tilde{b}z_{t}+c^{2}q_{t}%
\end{pmatrix}
\\
\quad \left( 
\begin{array}{c}
\dot{p}_{t} \\ 
\dot{q}_{t}%
\end{array}%
\right) &=&%
\begin{pmatrix}
0 \\ 
p_{t}z_{t}\left( 1-p_{t}\right) -\tilde{b}q_{t}%
\end{pmatrix}%
.
\end{eqnarray*}%
The following lemma is then obvious (only $y$ requires a computation, due to
the additional term in the Hamiltonian ODEs).

\begin{lemma}[Partial Hamiltonian Flow, correlated case]
\label{LemmaHamFlowCor}Consider the above Hamiltonian ODEs as initial value
problem, with initial data $\left( y_{0},z_{0}\right) =(0,0)$ and $(p,q_{0})$
and assume%
\begin{equation}
\chi _{p}^{2}:=c^{2}p\left( p-1\right) -\tilde{b}_{p}^{2}\geq 0.
\end{equation}%
Then the explicit solution for $z,p,q$ are then identical to the
uncorrelated case, one just has to replace $b$ by $\tilde{b}_{p}$
throughout. The explicit solution for $y$ is modified to 
\begin{equation}
y_{t}=\frac{q_{0}^{2}c^{2}}{8\chi _{p}^{3}}\left[ \left( c^{2}\left(
2p-1\right) -2\rho c\tilde{b}_{p}\right) \left( 2\chi _{p}t-\sin \left(
2\chi _{p}t\right) \right) +2\rho c\chi _{p}\left( 1-\cos \left( 2\chi
_{p}t\right) \right) \right] .  \label{yt_SScorel}
\end{equation}%
\bigskip\ 
\end{lemma}

In our explicit analysis of the uncorrelated case (more precisely, in
solving the coupled ODEs $\dot{z}_{t}=bz_{t}+c^{2}q_{t},\dot{q}%
_{t}=p_{t}z_{t}\left( 1-p_{t}\right) -bq_{t})$ we made use of the (model)
assumption $b\leq 0$, cf. (\ref{minusBTpos}). Conveniently, this remains
true when $\rho \in (-1,0]$. Indeed, the following lemma shows we must have $%
p\geq 0$, so that (with $\rho \leq 0,c>0$)%
\begin{equation}
\tilde{b}=b+\rho cp\leq 0\text{.}  \label{bTildeNeg}
\end{equation}

\begin{lemma}
\label{LemmaSSCorPpos}Let $\mathrm{a}>0$. Then $\Lambda \left( \mathrm{a}%
\right) =p\mathrm{a}$ and therefore $p\geq 0$.
\end{lemma}

\begin{proof}
We saw in the proof of lemma \ref{LemmaEnergySS_PosP} that, in the
uncorrelated case, as a direct consequence of the Hamiltonian ODEs,%
\begin{equation*}
p^{2}z_{t}^{2}+c^{2}q_{t}^{2}=2p\dot{y}_{t}+\partial _{t}\left(
z_{t}q_{t}\right) .
\end{equation*}%
The correlated case has the identical Hamiltonian ODEs provided we substitute%
\begin{equation*}
b\leftarrow \tilde{b}\text{ and }\dot{y}\leftarrow \dot{y}-\rho cz_{t}q_{t}.
\end{equation*}%
We therefore have%
\begin{eqnarray*}
\left\vert \mathrm{\dot{h}}_{0}\left( t\right) \right\vert ^{2} &=&\left(
p\,\,\,\,q_{t}\right) \left( 
\begin{array}{cc}
z^{2} & \rho cz \\ 
\rho cz & c^{2}%
\end{array}%
\right) \left( 
\begin{array}{c}
p \\ 
q_{t}%
\end{array}%
\right) =p^{2}z_{t}^{2}+c^{2}q_{t}^{2}+2\rho cpz_{t}q_{t} \\
&=&2p\left( \dot{y}_{t}-\rho cz_{t}q_{t}\right) +\partial _{t}\left(
z_{t}q_{t}\right) +2\rho cpz_{t}q_{t}=2p\dot{y}_{t}+\partial _{t}\left(
z_{t}q_{t}\right)
\end{eqnarray*}%
and then conclude with the boundary data, exactly as in lemma \ref%
{LemmaEnergySS_PosP}.
\end{proof}

As already noted, $\tilde{b}\leq 0$ allows to recycle all closed form
expressions for $z,q$ obtained in the uncorrelated case - it suffices to
replace $b$ by $\tilde{b}$. In particular, for some yet unknown $p,q_{0}$
which may and will depend on $\rho $, 
\begin{align}
z_{t}& =\frac{q_{0}c^{2}}{\chi _{p}}\sin \left( \chi _{p}t\right) ,  \notag
\\
q_{t}& =q_{0}\left( \cos \left( \chi _{p}t\right) -\frac{\tilde{b}}{\chi _{p}%
}\sin \left( \chi _{p}t\right) \right)  \notag
\end{align}%
where $\chi _{p}^{2}:=c^{2}p\left( p-1\right) -\tilde{b}^{2}$ is seen to be
positive as in the "uncorrelated" argument. Also, $q_{0}\neq 0$, seen as in
the "uncorrelated" case. Transversality, $q_{T}=0$, then implies%
\begin{equation}
\chi _{p}\cos \left( \chi _{p}T\right) -\tilde{b}\sin \left( \chi
_{p}T\right) =0.  \label{FromTransversalitySSCorl}
\end{equation}%
Introducing $r:=\chi _{p}T$ the gives the equation%
\begin{equation}
r\cot r=\left( b+\rho cp\right) T.  \label{rcotr}
\end{equation}%
On the other hand, from the very definition of $\chi _{p}$, we know%
\begin{equation}
\left( r/T\right) ^{2}=c^{2}p\left( p-1\right) -\left( b+\rho cp\right) ^{2}.
\label{rp_quadratic}
\end{equation}%
In the uncorrelated case, these two equations were effectively decoupled; in
particular, $r\cot r=bT$ lead to $r\in \left\{ r_{k}^{+}:k=1,2,\dots
\right\} \subset \left( 0,\infty \right) $, written in increasing order.
Since $p^{+}$ was seen to be monotonically increasing in $r$, cf. equation (%
\ref{p0explicit}), and we were looking for the minimal $p$, corresponding to
the minimal energy (cf. lemma \ref{LemmaSSCorPpos}), we were led to seek the
first positive root $r_{1}^{+}$. (In fact, $r_{1}^{+}\in (\pi /2,\pi )$ as
we will also find in the "correlated" discussion below.)

\bigskip

The correlated case is a little more complicated and we start in expressing $%
p$ in equation (\ref{rcotr}) in terms of $r$. Indeed, the quadratic equation
(\ref{rp_quadratic}) shows

\begin{equation}
p^{\pm }\left( r\right) \bigskip =\frac{1}{2\left( 1-\rho ^{2}\right) }%
\left\{ \left( 1+2\rho \frac{b}{c}\right) \pm \sqrt{\left( 1+2\rho \frac{b}{c%
}\right) ^{2}+4\left( 1-\rho ^{2}\right) \left[ \frac{b^{2}}{c^{2}}+\frac{%
r^{2}}{c^{2}T^{2}}\right] }\right\} ,\   \label{SScor_pPlus}
\end{equation}%
where $p^{-}\left( r\right) \bigskip <0$ (and hence can be ignored in view
of lemma \ref{LemmaSSCorPpos}) and $p^{+}\left( r\right) \bigskip >0$. We
now look for $r$ which satisfies the equation%
\begin{equation*}
r\cot r=\left( b+\rho cp^{+}\left( r\right) \bigskip \right) T
\end{equation*}%
It is elementary to see that $r\cot r$ is non-negative on $\left[ 0,\pi /2%
\right] $ and then maps $[\pi /2,\pi )$ strictly monotonically to $(-\infty
,0]$. On the other hand, the map $r\mapsto $ $(b+\rho cp^{+}\left( r\right)
\bigskip )T$ is $\leq 0$ for all $r$; in particular, there will be a first
intersection with the graph of $r\mapsto $ $r\cot r$ in $[\pi /2,\pi )$, say
at $r=r_{1}^{+}$. Since $p^{+}\left( r\right) \bigskip $ is plainly strictly
increasing in $r$, the minimal $p$ must equal to%
\begin{equation*}
p_{1}^{+}:=p^{+}\left( r_{1}^{+}\right) \text{.}
\end{equation*}%
We then proceed as in the uncorrelated case, and determine $q_{0}$ from the
boundary condition $y_{T}=\mathrm{a}>0$ where $y$ is now given by (\ref%
{yt_SScorel}). This leads to $q_{0}\in \left\{
q_{0,1}^{+},q_{0,1}^{-}\right\} $ where 
\begin{equation*}
q_{0,1}^{\pm }=\pm \frac{2}{c}\sqrt{\frac{2r^{3}\,\mathrm{a}}{T^{3}\left(
\left( c^{2}\left( 2p-1\right) -2\rho c\tilde{b}\right) \left( 2r-\sin
\left( 2r\right) \right) +2\rho cr/T\left( 1-\cos \left( 2r\right) \right)
\right) }}
\end{equation*}%
where $r=r_{1}^{+}$ and $p=p_{1}^{+}.$ Again, we have \textit{two}
minimizing controls, $\mathcal{K}_{\mathrm{a}}^{\min }=\left\{ \mathrm{h}%
_{0}^{+},\mathrm{h}_{0}^{-}\right\} $. We now have%
\begin{equation}
\mathrm{\dot{h}}_{0}\left( t\right) =\left( 
\begin{array}{cc}
z_{t}\sqrt{1-\rho ^{2}} & 0 \\ 
\rho z_{t} & c%
\end{array}%
\right) \left( 
\begin{array}{c}
p \\ 
q_{t}%
\end{array}%
\right)  \label{h0explicit_SScor}
\end{equation}%
instead of (\ref{SSh0dotExplicit}) and of course lemma \ref{LemmaHamFlowCor}
implies that $z_{t}$ and $q_{t}$ are fully and explicitly determined for
each choice of $\left( p,q_{0}\right) $. In particular for $\left(
p,q_{0}\right) \leftarrow \left( p_{1}^{+},q_{0,1}^{+}\right) $ resp. $%
\left( p_{1}^{+},q_{0,1}^{-}\right) $ we so obtain $\mathrm{h}_{0}^{+}$
resp. $\mathrm{h}_{0}^{-}$ which can be written explicitly by simple
substitution. Moreover, and again as in the uncorrelated case,%
\begin{equation}
\Lambda \left( \mathrm{a}\right) =\frac{1}{2}\Vert \mathrm{h}_{0}^{+}\Vert
_{H}^{2}=\frac{1}{2}\Vert \mathrm{h}_{0}^{-}\Vert _{H}^{2}=p_{1}^{+}\mathrm{a%
}
\end{equation}%
and upon taking $\mathrm{a}=1$ we have computed the leading order constant%
\begin{equation*}
c_{1}=\Lambda \left( 1\right) =p_{1}^{+}\equiv p^{+}\left( r_{1}^{+}\right)
\bigskip
\end{equation*}%
where we recall that $r_{1}^{+}$ is the first intersection point of $%
r\mapsto $ $r\cot r$ with $(b+\rho cp^{+}\left( r\right) \bigskip )T$ and $%
p^{+}\left( \cdot \right) \bigskip $ was given in (\ref{SScor_pPlus}).

\bigskip

At last, we turn to the computation of the second-order exponential
constant, $c_{2}$. {}\ As in the uncorrelated case, we ease notation by
writing $\left( p,q_{0}^{\pm }\right) $ instead of $\left(
p_{1}^{+},q_{0,1}^{+}\right) $ resp.$\left( p_{1}^{+},q_{0,1}^{-}\right) $
for the rest of this section. Again, we have to consider ODEs for $\left( 
\hat{Y}_{t},\hat{Z}_{t}\right) $, for each minimizer, $\mathrm{h}%
_{0}^{+}=(h_{0,\cdot }^{+,1},h_{0,\cdot }^{+,2})$ and $\mathrm{h}%
_{0}^{-}=(h_{0,\cdot }^{+,1},-h_{0,\cdot }^{+,2})$. Recall from (\ref%
{h0explicit_SScor}) that, with $\bar{\rho}=\sqrt{1-\rho ^{2}}$, 
\begin{equation*}
\mathrm{\dot{h}}_{0}^{+}\left( t\right) =\left( 
\begin{array}{c}
p\bar{\rho}z_{t}^{+} \\ 
\rho pz_{t}^{+}+cq_{t}^{+}%
\end{array}%
\right) ;
\end{equation*}%
where $\left( \cdot \right) ^{\pm }$ indicates the chosen sign of $q_{0}\in
\left\{ q_{0,1}^{+},q_{0,1}^{-}\right\} $ which determines the choice of
minimizer. We first determine $\hat{Y}_{T}=\hat{Y}_{T}\left( \mathrm{h}%
_{0}^{+}\right) $ from the ODE 
\begin{eqnarray*}
\frac{d}{dt}\left( 
\begin{array}{c}
\hat{Y}_{t} \\ 
\hat{Z}_{t}^{2}%
\end{array}%
\right) &=&\left\{ \left( 
\begin{array}{cc}
0 & -z_{t}^{+} \\ 
0 & b%
\end{array}%
\right) +\left( 
\begin{array}{cc}
0 & \bar{\rho} \\ 
0 & 0%
\end{array}%
\right) \dot{h}_{0,t}^{+,1}+\left( 
\begin{array}{cc}
0 & \rho \\ 
0 & 0%
\end{array}%
\right) \dot{h}_{0,t}^{+,2}\right\} \left( 
\begin{array}{c}
\hat{Y}_{t} \\ 
\hat{Z}_{t}%
\end{array}%
\right) +\left( 
\begin{array}{c}
0 \\ 
a%
\end{array}%
\right) \\
&=&\left( 
\begin{array}{cc}
0 & \left( p-1\right) z_{t}^{+}+\rho cq_{t}^{+} \\ 
0 & b%
\end{array}%
\right) \left( 
\begin{array}{c}
\hat{Y}_{t} \\ 
\hat{Z}_{t}%
\end{array}%
\right) +\left( 
\begin{array}{c}
0 \\ 
a%
\end{array}%
\right) \\
\,\,\text{with }\left( 
\begin{array}{c}
\hat{Y}_{0} \\ 
\hat{Z}_{0}^{2}%
\end{array}%
\right) &=&\left( 
\begin{array}{c}
0 \\ 
\sigma _{0}%
\end{array}%
\right) .
\end{eqnarray*}%
This already shows that we have the identical (closed form) ODE solution for 
$\hat{Z}_{t}$ as in the uncorrelated case. On the other hand, the form of $%
\hat{Y}_{T}$ now exhibits an additional term as is seen in%
\begin{equation*}
\hat{Y}_{T}=\left( p-1\right) \int_{0}^{T}z_{t}^{+}\hat{Z}_{t}dt+\rho
c\int_{0}^{T}q_{t}^{+}\hat{Z}_{t}dt.
\end{equation*}%
Since $q_{t}^{+}$ is essentially of the same trigonometric form as $%
z_{t}^{+} $, it is clear that the explicit computations of the uncorrelated
case extend. In the end, one finds without too much difficulties%
\begin{equation*}
c_{2}^{+}=\Lambda ^{\prime }\left( 1\right) \times \hat{Y}_{T}\left( \mathrm{%
h}_{0}^{+}\right) =q_{0}^{+}\left\{ \sigma _{0}+a\frac{\tan \left( \chi
_{p}T/2\right) }{\chi _{p}}\right\} .
\end{equation*}%
A similar computation along $\mathrm{h}_{0}^{-}$ gives $c_{2}^{+}=\Lambda
^{\prime }\left( 1\right) \times \hat{Y}_{T}\left( \mathrm{h}_{0}^{-}\right) 
$ in explicit form and $c_{2}$ is identified as $\max \left(
c_{2}^{+},c_{2}^{-}\right) $.

\subsubsection{Checking non-degeneracy, zero and non-zero correlation}

We now check the non-degeneracy conditions, contained in assumptions 
(i)-(iii) of theorem \ref{thm:MainThm}, which of course is the ultimate justification that an expansion
of the form (\ref{GSexpansion}) with the constants computed above holds
true. Again, focus is on the case of correlation parameter $\rho \in (-1,0]$%
. We saw in the previous sections (for $\rho =0$, then $\rho \leq 0$) that $%
\#K_{\mathrm{a}}^{\min }=\#\left\{ \mathrm{h}_{0}^{+},\mathrm{h}%
_{0}^{-}\right\} =2$, whenever $\mathrm{a}>0$. (In fact, we apply this with $%
\mathrm{a}=1$.)

Secondly, a look at (\ref{SSeps0}) reveals that the \textit{degenerate region%
} is $\left\{ \left( y,z\right) :z=0\right\} $, the complement of which is
elliptic. Clearly, no controlled path which reaches $y_{T}=\mathrm{a}>0$ can
stay in the degenerate region for all times $t\in \left[ 0,T\right] $; after
all, this would entail $dy=0$ and hence $y_{T}=0$. We conclude the any ODE
solution driven by $h\in \mathcal{K}_{a}$ must intersect the region of
ellipticity; but this already implies non-degeneracy of the corresponding
(deterministic) Malliavin covariance matrix.

At last, we check non-focality and focus on $\mathrm{h}_{0}^{+}$, the other
case being similar. We have to check non-degeneracy of the Jacobian of the
map $\pi H_{0\leftarrow T}\left( \mathrm{a},\cdot ;\ast ,0\right) $,
evaluated at $\cdot =z_{T},\ast =p_{T}$ after differentiation, where $%
z_{T},p_{T}$ are obtained form the Hamiltonian flow at time $T$, cf. lemma %
\ref{LemmaHamFlowCor}, with time $0$ initial data $\left(
0,0;p_{1}^{+},q_{0,1}^{+}\right) $.

\bigskip

With some abuse of notation, write%
\begin{equation*}
\left( 
\begin{array}{c}
y_{0} \\ 
z_{0}%
\end{array}%
\right) \equiv \left( 
\begin{array}{c}
y_{0}\left( z,p\right) \\ 
z_{0}\left( z,p\right)%
\end{array}%
\right) \equiv \pi H_{0\leftarrow T}\left( \mathrm{a},z;p,0\right) .
\end{equation*}

Our non-degeneracy condition requires us to show that 
\begin{equation}
\left. \det \left( 
\begin{array}{cc}
\frac{\partial y_{0}}{\partial p} & \frac{\partial y_{0}}{\partial z} \\ 
\frac{\partial z_{0}}{\partial p} & \frac{\partial z_{0}}{\partial z}%
\end{array}%
\right) \right\vert _{\ast }\neq 0  \label{detNonFocalNZ}
\end{equation}%
where $\left( ...\right) |_{\ast }$ indicates evaluation $\left( ...\right)
|_{\left( p,z\right) =\left( p^{+},z_{T}\right) }$ in the sequel. This
implies in particular that all expressions which are formulated in terms of
the solutions to the Hamiltonian flows, reduced to the corresponding
expressions identified in proposition \ref{PropSolSSHamODE_BVP}, for $\rho
=0 $, resp. in section \ref{SecSSnonzerocorl} for $\rho \leq 0$. For
instance, $\left( y_{0},z_{0}\right) |_{\ast }=\left( 0,0\right)
,y_{T}|_{\ast }=\mathrm{a},z|_{\ast }=z_{T}\neq 0$, $\chi _{p}T|_{\ast }\in
\lbrack \pi /2,\pi )$ and so.

\bigskip

Since $\left( z_{\cdot },q_{\cdot }\right) $ solves a linear ODE, we can
compute 
\begin{eqnarray*}
z_{0}\left( z,p\right) &=&\left( 
\begin{array}{cc}
1 & 0%
\end{array}%
\right) \,e^{-T\left( 
\begin{array}{cc}
\tilde{b}_{p} & c^{2} \\ 
p\left( 1-p\right) & -\tilde{b}_{p}%
\end{array}%
\right) }\left( 
\begin{array}{c}
z \\ 
0%
\end{array}%
\right) \\
&=&\frac{z}{\chi _{p}}\left( \chi _{p}\cos \left( \chi _{p}T\right) -\tilde{b%
}_{p}\sin \left( \chi _{p}T\right) \right) .
\end{eqnarray*}%
We first note that $\partial z_{0}/\partial z|_{\ast }$ is zero; indeed,
this follows from (\ref{FromTransversalitySSCorl}). Our next claim is $%
\partial y_{0}/\partial z|_{\ast }\neq 0$. Indeed, from the structure of the
Hamilton ODEs, 
\begin{equation*}
y_{0}-\mathrm{a}=-\int_{0}^{T}\dot{y}_{t}dt=z^{2}\left( ...\right)
\end{equation*}%
where $\left( \cdots \right) $ does not depend on $z$. As a result $\partial
y_{0}/\partial z|_{\ast }=2z\left( ...\right) |_{\ast }=2\frac{y_{0}-\mathrm{%
a}}{z}|_{\ast }=-2\mathrm{a/}z_{T}\neq 0.$

It remains to check that $\partial z_{0}/\partial p|_{\ast }\neq 0$. To this
end, recall, as a consequence of the transversality condition, see (\ref%
{FromTransversalitySSCorl}), that $\chi _{p}\cos \left( \chi _{p}T\right) -%
\tilde{b}_{p}\sin \left( \chi _{p}T\right) |_{\ast }=0$. It follows that%
\begin{equation*}
\partial z_{0}/\partial p|_{\ast }=\left\{ \frac{z}{\chi _{p}}\frac{\partial 
}{\partial p}\left( \chi _{p}\cos \left( \chi _{p}T\right) -\tilde{b}%
_{p}\sin \left( \chi _{p}T\right) \right) \right\} _{\ast }
\end{equation*}%
and since $z/\chi _{p}|_{\ast }\neq 0$, it will be enough to show\ (strict)
negativity of $\frac{\partial }{\partial p}\left( ...\right) |_{\ast }$
above. By scaling, there is no loss of generality in taking $T=1$ and we
shall do so from here on. Then 
\begin{eqnarray*}
&&\frac{\partial }{\partial p}\left( \chi _{p}\cos \left( \chi _{p}\right) -%
\tilde{b}_{p}\sin \left( \chi _{p}\right) \right) \\
&=&\chi _{p}^{\prime }[\left( 1-\tilde{b}_{p}\right) \cos \left( \chi
_{p}\right) -\chi _{p}\sin \left( \chi _{p}\right) ]-\rho c\sin \left( \chi
_{p}\right) .
\end{eqnarray*}%
Since $\tilde{b}_{p}|_{\ast }\leq 0$ and $\chi _{p}|_{\ast }\in \lbrack \pi
/2,\pi )$ we see that $\left[ ...\right] |_{\ast }<0$. Given that $\chi
_{p}^{\prime }|_{\ast }>0$, this already settles the negativity claim in the
zero-correlation case. In the case $-1<\rho <0$, we use (\ref%
{FromTransversalitySSCorl}) to write%
\begin{eqnarray*}
&&\frac{\partial }{\partial p}\left( \chi _{p}\cos \left( \chi _{p}\right) -%
\tilde{b}_{p}\sin \left( \chi _{p}\right) \right) |_{\ast } \\
&=&\chi _{p}^{\prime }[\left( 1-\tilde{b}_{p}\right) \frac{\tilde{b}_{p}\sin
\left( \chi _{p}\right) }{\chi _{p}}-\chi _{p}\sin \left( \chi _{p}\right)
]-\rho c\sin \left( \chi _{p}\right) |_{\ast }.
\end{eqnarray*}%
After division by $\sin \left( \chi _{p}\right) /\chi _{p}|_{\ast }>0$, we
have, using $\tilde{b}_{p}=b+\rho cp\leq 0,$ $b\leq 0$ and again $\chi
_{p}^{\prime }|_{\ast }>0$, 
\begin{eqnarray*}
&&\chi _{p}^{\prime }[\left( 1-\tilde{b}_{p}\right) \tilde{b}_{p}-\chi
_{p}^{2}]-\rho c\chi _{p}|_{\ast } \\
&\leq &\chi _{p}^{\prime }[\left( 1-\rho cp\right) \rho cp-\chi
_{p}^{2}]-\rho c\chi _{p}|_{\ast } \\
&\leq &-\rho c\left( \chi _{p}-p\chi _{p}^{\prime }\right) |_{\ast }.
\end{eqnarray*}%
With \thinspace $-\rho c>0$, it will then be sufficient to show strict
negativity of $\chi _{p}-p\chi _{p}^{\prime }|_{\ast }$. To this end note
that the definition, $\chi _{p}^{2}=c^{2}p\left( p-1\right) -\tilde{b}^{2}$,
implies 
\begin{eqnarray*}
2\chi _{p}\chi _{p}^{\prime } &=&c^{2}\left( 2p-1\right) -2\tilde{b}\left(
\rho c\right) \\
\chi _{p}p\chi _{p}^{\prime } &=&c^{2}p\left( p-1/2\right) -\tilde{b}\left(
\rho cp\right) \\
&=&\chi _{p}^{2}+\frac{c^{2}p}{2}+b\tilde{b}>\chi _{p}^{2}
\end{eqnarray*}%
whenever $c^{2}p/2+b\tilde{b}>0$ which is surely the case upon evaluation $%
...|_{\ast }$.

We conclude that $\partial z_{0}/\partial p|_{\ast }\neq 0$, and then
validity of (\ref{detNonFocalNZ}), for any parameter set $\rho \in
(-1,0],b\leq 0,c>0,T>0$. In other words, we have completed the check of our
non-degeneracy condition.

\subsection{Comments on Heston \protect\cite{He} and Lions--Musiela 
\protect\cite{LM}}

We recall from \cite{GuSt, FGGS} that the density of log-stock price $Y_{T}$
in the Heston model, 
\begin{eqnarray*}
dY &=&-V/2+\sqrt{V}dW^{1},\,\,X\left( 0\right) =x_{0}=0 \\
dV &=&\left( a+bV\right) dt+c\sqrt{V}dW^{2},\,\,\,V\left( 0\right) =v_{0}>0,
\end{eqnarray*}%
with $a\geq 0,b\leq 0,c>0$ and correlation $\rho \in (-1,0]$ has the form%
\begin{equation*}
f\left( y\right) =e^{-c_{1}y}e^{c_{2}\sqrt{y}}y^{-3/4+a/c^{2}}\left(
c_{3}+O\left( 1/\sqrt{y}\right) \right) \text{ as }y\rightarrow \infty ;
\end{equation*}%
with explicitly computable $c_{1}=C_{1}\left( b,c,\rho ,T\right) $ and $%
c_{2}=\sqrt{v_{0}}\times C_{2}\left( b,c,\rho ,T\right) $, both do not
depend on $a$. While \textbf{scaling} with $\theta =2$, 
\begin{equation*}
Y_{\varepsilon }:=\varepsilon ^{2}Y,\,\,\,\,\,V_{\varepsilon }:=\varepsilon
^{2}V
\end{equation*}%
indeed yields a small noise problem, namely%
\begin{eqnarray*}
dY^{\varepsilon } &=&-V^{\varepsilon }/2+\sqrt{V^{\varepsilon }}\varepsilon
dW^{1},\,\,\,X\left( 0\right) =x_{0}=0 \\
dV^{\varepsilon } &=&\left( a\varepsilon ^{2}+bV^{\varepsilon }\right) dt+c%
\sqrt{V^{\varepsilon }}\varepsilon dW^{2},\,\,\,V\left( 0\right)
=v_{0}\varepsilon ^{2}>0.
\end{eqnarray*}%
The algebraic factor $y^{-3/4+a/c^{2}}$ in the above expansion then
contradicts the expected factor; cf. (\ref{densityTheta})%
\begin{equation*}
y^{\frac{1}{\theta }-1}=y^{-1/2}.
\end{equation*}%
There is no contradiction here, of course. Rather, we see an explicit
example where "formal" application of a theorem to a model which is short of
the required regularity leads to wrong conclusion (at least at the fine
level of algebraic factors). Remark that one can trace the origin of this
unexpected $y^{-3/4+a/c^{2}}$ factor to the behaviour of the one-dimensional
variance process $V;$ also known as Feller - or Cox-Ingersoll-Ross\
diffusion. Curiously then \textit{even a large deviation principle for }$%
V^{\varepsilon }$\textit{\ as given above presently lacks justification,}
despite the recent advances in \cite{Donati},~\cite{Baldi}. Clearly then, we
are not anywhere near in obtaining the Heston tail result of \cite{GuSt,
FGGS} with the present methods.

However, in the special case when $a=c^{2}/4$ it is an easy exercise to see
that the Heston model can be realized as Stein-Stein model (take $V=Z^{2}$,
where $Z$ is the volatility component of the Stein-Stein model), the
resulting expressions are then seen to be consistent with those obtained in 
\cite{FGGS} and, in particular, $y^{-3/4+a/c^{2}}=y^{-1/2}$.

\bigskip

Another class of non-smooth, non-affine stochastic vol model with $"\theta
=2"$-scaling, introduced by\ Lions-Musiela \cite{LM}. For $\delta \in \left[
1/2,1\right] $ and $\gamma =1-\delta $ they consider the $2$-dimensional
diffusion%
\begin{eqnarray*}
dY &=&-\frac{1}{2}Z^{2\delta }dt+Z^{\delta }d\tilde{W}_{1},\,\,\,Y_{0}=0 \\
dZ &=&bZdt+cZ^{\gamma }dW_{2},\,\,\,\,Z\left( 0\right) =z_{0}>0.
\end{eqnarray*}%
And indeed with $Y_{\varepsilon }=\varepsilon ^{2}Y$ and $Z_{\varepsilon
}=\varepsilon ^{1/\delta }Z$ this becomes a small noise problem;%
\begin{eqnarray*}
dY_{\varepsilon } &=&-\frac{1}{2}Z_{\varepsilon }^{2\delta
}dt+Z_{\varepsilon }^{\delta }\varepsilon dW,\,\,\,\,Y_{\varepsilon }\left(
0\right) =0 \\
dZ_{\varepsilon } &=&bZ_{\varepsilon }dt+cZ_{\varepsilon }^{\gamma
}\varepsilon dZ,\,\,Z_{\varepsilon }\left( 0\right) =\varepsilon ^{1/\delta
}z_{0}.
\end{eqnarray*}%
In their paper they establish exponential moments of $Y_{T}$. It is tempting to use corollary \ref%
{CorTail}, at least to leading large deviation order, to obtain the
exponential tail of $Z$ for models that scale with $\theta =2$. Of course,
as was discussed in the Heston case, such a "formal" application can be
wrong. Further work, building on \cite{Donati},~\cite{Baldi}, will be
necessary to deal with such degenerate models directly.

\end{document}